\documentclass[smallextended,envcountsect]{svjour3} 

\smartqed 

\usepackage{graphicx}
\usepackage{a4wide}
\usepackage{enumitem}
\usepackage{color}
\usepackage{marvosym}
\usepackage{amsmath,amssymb,url}
\usepackage{algorithmicx}
\usepackage{algpseudocode}
\usepackage[Algorithm,ruled]{algorithm}

\usepackage{subfigure}

\journalname{---}

\usepackage{cite} 

\newtheorem{assumption}{Assumption}
\graphicspath{{../}} 
\newcommand{\norm}[1]{\textstyle{\Vert} #1 \textstyle{\Vert}_2}

\newcommand{\dom}[1]{\mathrm{dom}\, #1} 

\newcommand{\col}[1]{\left\{#1\right\}}

\newcommand{\tol}{\texttt{Tol}}
\newcommand{\inner}[2]{\langle#1,#2\rangle}

\newcommand{\ind}{\iota}

\newcommand{\argmin}{\operatorname{argmin}}

\newcommand{\epi}[1]{\operatorname{epi}#1}

\newcommand{\fc}[2]{: #1 \rightarrow #2}

\newcommand{\tto}{\rightrightarrows}

\renewcommand{\Re}{\mathbb{R}}
\newcommand{\proj}{\mathtt{P}}

\newcommand{\modelh}{\mathfrak{h}}

\newcommand{\mybox}{{\hfill $\square$}}

\definecolor{gris}{gray}{0.6}

\newcommand{\nats}{\mathbb{N}}
\newcommand{\reals}{\mathbb{R}}
\newcommand{\Reals}{\overline{\mathbb{R}}}
\newcommand{\nInnLim}{\mathop{\rm LimInn}\nolimits}
\newcommand{\nOutLim}{\mathop{\rm Lim\hspace{-0.01cm}Out}\nolimits}

\newcommand{\nargmax}{\mathop{\rm argmax}\nolimits}
\newcommand{\nargmin}{\mathop{\rm argmin}\nolimits}

\newcommand{\nliminf}{\mathop{\rm liminf}\nolimits}
\newcommand{\cN}{{\cal N}}
\newcommand{\grill}{{\scriptscriptstyle\#}}
\DeclareMathOperator{\con}{con}
\DeclareMathOperator{\dist}{dist}
\DeclareMathOperator{\gph}{gph}
\newcommand{\sto}{\,{\lower 1pt\hbox{$\rightarrow$}}\kern -10pt
     \hbox{\raise 4pt \hbox{$\, \scriptstyle s$}}\hskip7pt}
\newcommand{\eto}{\,{\lower 1pt\hbox{$\rightarrow$}}\kern -10pt
     \hbox{\raise 4pt \hbox{$\, \scriptstyle e$}}\hskip7pt} 
\newcommand{\gto}{\,{\lower 1pt\hbox{$\rightarrow$}}\kern -10pt
     \hbox{\raise 4.5pt \hbox{$\, \scriptstyle g$}}\hskip7pt}     
\def\Nto{\,{\raise 1pt\hbox{$\rightarrow$}}\kern -12pt
     \hbox{\lower 3pt \hbox{$\, \scriptstyle N$}}\hskip7pt}       
\newcommand{\nnmin}{\mathop{\rm minimize}}

\begin{document}
\title{Composite Optimization using Local Models and Global Approximations}

\author{Welington de Oliveira$^1$ \and Johannes O. Royset$^2$}
\authorrunning{W. de Oliveira and J.O. Royset}

\institute{$1$ MINES Paris, Universit\'e PSL, Centre de Math\'ematiques Appliqu\'ees (CMA), Sophia Antipolis
06904, France \\
 \email{\url{welington.oliveira@minesparis.psl.eu}}\\
$2$ University of Southern California, Daniel J. Epstein Dep. of Industrial \& Systems Eng., Los Angeles, 90089, USA\\
\email{\url{royset@usc.edu}}}

\date{\today}

\maketitle

\begin{abstract}
This work presents a unified framework that combines global approximations with locally built models to handle challenging nonconvex and nonsmooth composite optimization problems, including cases involving extended real‑valued functions. We show that near‑stationary points of the approximating problems converge to stationary points of the original problem under suitable conditions. Building on this, we develop practical algorithms that use tractable convex master programs derived from local models of the approximating problems. The resulting double-loop structure improves global approximations while adapting local models, providing a flexible and implementable approach for a wide class of composite optimization problems. It also lays the groundwork for new algorithmic developments in this domain.
\end{abstract}

\keywords{Composite Optimization \and Nonsmooth Optimization \and Nonconvex Optimization \and Variational Analysis}
\subclass{49J52 \and 49J53 \and  49K99 \and 90C26}

\section{Introduction}\label{sec:introduction}

Algorithms for continuous optimization mainly fall in one of two categories: {\em Model-based algorithms} proceed by iteratively constructing and solving {\em local} models, while {\em consistent approximation algorithms} construct one or more {\em global} approximations before optimizing by existing methods. Examples of the first category include gradient descent, Newton's, and cutting-plane methods as well as the large variety of bundle, trust-region, and surrogation methods. With some exceptions, model-based algorithms construct a {\em convex} model locally near a current iterate using the function value and a subgradient at the iterate and, possibly, also information from prior iterates and nearby points. In our assessment, this approach has been successful broadly and has resulted in some of the leading algorithms for continuous optimization; see, e.g.,  \cite{ConnGouldToint.00, NocedalWright.06,Ruszczynski.06} for details about classical algorithms. 

Smoothing algorithms rely on consistent approximations as they construct global approximations apriori that can be solved using gradient-based algorithms; examples appear in \cite{PeeRoyset2011,Chen.12,BurkeHoheisel.13}. Penalty methods for constrained optimization and sample average approximations for stochastic problems furnish additional examples. For the latter, the global approximations consist of sample averages of expectation functions; see, e.g., \cite{Shapiro2021a}. The analysis of consistent approximation algorithms tends to be agnostic to {\em how} its global approximations are  solved eventually \cite{Royset2023a}, \cite[\S\, 4.C]{Royset_Wets_2021}, \cite[\S\, 3.3]{Polak1997}. While this flexibility can be advantageous, it may leave crucial implementation details about subroutines, warm-starts, and tolerances unspecified. 

Although augmented Lagrangian methods can be viewed as seamlessly coupling a local model (expressed in terms of multipliers) with global approximations in the form of quadratic penalties, there is no general algorithmic framework tightly integrating the two concepts. In this paper, we develop such a framework in the context of nonconvex and nonsmooth, composite optimization, while making separate contributions to both model-based methods and consistent approximation algorithms. The framework allows us to construct and analyze implementable algorithms for a broad class of problems, while retaining significant flexibility for further specialization. 

Given a nonempty and closed set $X\subset \reals^n$, a proper, lower semicontinuous (lsc), and convex function $h:\reals^m\to \Reals = [-\infty,\infty]$, a smooth (i.e., $C^1$) function $f_0:\reals^n\to \reals$, and a locally Lipschitz continuous (lLc) mapping $F:\reals^n\to \reals^m$, we consider the composite optimization problem
\begin{equation}\label{pbm}
\nnmin_{x\in X}\; f_0(x) + h\big(F(x)\big),
\end{equation} 
which we refer to as the {\em actual problem}. A portion of the analysis (Section~\ref{sec:global}) takes place under the stated assumptions, with some moderate restrictions (mainly convexity of $X$ and $f_0$) entering in later sections to ensure implementable operations in the algorithms. 

Many applications result in the form \eqref{pbm}; see, e.g., \cite{AravkinBurkePillonetto.14,CuiPangSen.18,DuchiRuan.19,DavisDrusvyatskiyPaquette.20,CharisopoulosDavisDiazDrusvyatskiy.21,ByundeOliveiraRoyset.23,Ferris_Huber_Royset_2024,Royset.25,Sempere_Oliveira_Royset_2025}. Applications include phase retrieval, robust estimation, Kalman smoothing, sensing, risk management, and engineering design. The form covers nonlinear programming, goal programming, difference-of-convex (dc) optimization, and convex-composite optimization. It extends beyond optimization to variational inequalities because these can be written as nonsmooth optimization problems \cite{Ferris_Huber_Royset_2024,Royset.25}. The actual problem \eqref{pbm} brings forth convexity, smoothness, and Lipschitz properties, which are keys to variational analysis and efficient algorithms. Algebraic modeling languages also leverage such properties as they provide means to describe optimization models in a format readable by solvers \cite{Ferris_Huber_Royset_2024}. The importance of the composite form $h\circ F$, with $h$ convex and $F$ smooth, was recognized already in \cite{Powell.78a,Powell.78b}. 

Two factors hamper the development of implementable algorithms for the actual problem: the extended real-valuedness of $h$, which is essential for modeling nonlinear constraints, and the nonsmoothness and nonconvexity of the component functions of $F$. If $h$ is real-valued and convex and $F$ is twice smooth (i.e., $C^2$), then there are several solution approaches including proximal composite methods that linearize $F$ and solve regularized convex master problems. These methods can be traced back to \cite{Fletcher1982,Powell1983,Burke1985}; see also \cite{Powell1984,Yuan1985,BurkeFerris1995} for trust-region versions and \cite{Drusvyatskiy2019} for details about rates of convergence under the assumption that $h$ and $\nabla F$ are Lipschitz continuous. Under the additional assumption that $h$ is positively homogeneous, a bundle method is also available \cite{Sagastizabal_2013}. When the feasible set is relatively simply, subgradient-type algorithms may also apply; see, for example, \cite{LiCui2025a} for recent advances and unifying analysis. While retaining a twice smooth $F$, \cite{LewisWright.16} allows for extended real-valued $h$ (which actually does not have to be convex either; subdifferential regularity and prox-boundedness suffice) but the resulting algorithms appear challenging to implement for certain $h$. Other efforts involving more general composite problems such as those defined by the broad class of quasi-difference-convex functions include \cite{PangPeng2025}. 

Chapter 7 in \cite{Cui_Pang_2022} puts forth the broad framework of surrogation, which is a general approach for constructing local, upper bounding models that in turn are solved using existing (convex) optimization algorithms. This is also the setting in \cite{Syrtseva_2024} and \cite{Sempere_Oliveira_Royset_2025}. With  focus on local models, such a framework does not inherently involve global approximations. A step in that direction is provided by \cite{Liu2022a}, which replaces nested expectation functions with sample averages. Further coupling of local models and global approximations appears in \cite{LiCui2025b}, while considering a composition of a sum of univariate extended real-valued functions with an inner function, possibly not lLc, and establishing epi-convergence of the approximations. 

The general framework of consistent approximations put forth in \cite{Royset2023a} conceptually handles extended real-valued, proper, lsc, and convex $h$ by considering real-valued approximations. However, it implicitly assumes that multipliers produced by an algorithm have cluster points, but this cannot be taken for granted when $h$ is extended real-valued. Thus, the earlier effort does not fully address nonlinearly constrained problems. Implementation details are also limited to one specific proximal composite method for solving global approximations, while relying on twice smooth approximations of $F$ and $f_0$ being absent.

In this paper, we construct a framework that combines local models and global approximations to produce implementable algorithms for \eqref{pbm} under relaxed assumptions. The global approximations amount to replacing $h$ by another proper, lsc, and convex function $h^\nu:\reals^m\to \Reals$, $F$ by another lLc mapping $F^\nu:\reals^n\to \reals^m$, and $f_0$ by another smooth function $f_0^\nu:\reals^n\to \reals$. When $X$ is  convex, it can be replaced by a convex set $X^\nu\subset\reals^n$. This produces the \emph{approximating problems} for $\nu\in \nats = \{1, 2, \dots\}$:
\begin{equation}\label{pbm-a}
\nnmin_{x\in X^\nu}\; f_0^\nu(x) + h^\nu\big(F^\nu(x)\big),
\end{equation}
which presumably have computational advantages over the actual problem \eqref{pbm}. For example, $h^\nu$ could be a continuous or smooth approximation of $h$ and $F^\nu$ might be a smooth approximation of $F$; see \cite{ErmolievNorkinWets.95,ChenMangasarian.95,Chen.12,BurkeHoheisel.13,BurkeHoheisel.17} and below for concrete examples of smoothing techniques. The set $X^\nu$ could be a polyhedral approximation of $X$ and $f_0^\nu$ might represent vanishing regularization. 

As a first contribution, we show that any cluster point of a sequence of near-stationary points of the approximating problems \eqref{pbm-a}, as $\nu\to \infty$, must be a stationary point of the actual problem \eqref{pbm} under a qualification. In contrast to \cite{Royset2023a}, we include functions $f_0,f_0^\nu$ and dispose of an assumption of bounded multipliers. Thus, the framework handles nonlinearly constrained problems by letting $h$ assign infinite penalties at infeasible points. We also provide additional flexibility in the construction of optimality conditions defining stationarity for the problems. 

As a second contribution, we propose implementable algorithms based on local models for computing near-stationary points of an approximating problem \eqref{pbm-a}. While we assume that $X^\nu$ and $f_0^\nu$ are convex, the approximating problems tend to remain nonsmooth and nonconvex making local models essential. We construct local models of $F^\nu$ through linearization, which produces convex master problems, and more refined local models of $h^\nu$ to reduce the computational cost of solving these master problems. We extend the analysis of the composite bundle algorithm of \cite{Sagastizabal_2013} by neither assuming a positively homogeneous $h^\nu$ nor a twice smooth $F^\nu$; it suffices for $F^\nu$ to be $L$-smooth. We allow for convex $X^\nu$, convex and smooth $f_0^\nu$, and more general local models for $h^\nu$, which represent further extensions beyond \cite{Sagastizabal_2013}. In particular, we address situations where $h^\nu$ is in of itself a composite function of the form $h^\nu = h_0^\nu \circ H^\nu$ as motivated by \cite{Ferris_Huber_Royset_2024}. This offers opportunities for constructing specialized local models. 

These contributions allow us to analyze {\em double-loop algorithms} for the actual problem \eqref{pbm} of the form: 
\begin{align}
&\text{For each $\nu= 0, 1, \dots$ do:}\nonumber\\ 
& ~~~~\text{Compute a near-stationary point of \eqref{pbm-a} by iterating:}\nonumber\\
& ~~~~\text{For each $k = 0, 1, \dots$ do:}\nonumber\\
& ~~~~~~~~\text{Construct local models $h_k^\nu, F_k^\nu$ of $h^\nu, F^\nu$ at a current point $x_k^\nu$ and solve, with $t_k>0$,}\nonumber\\ 
& ~~~~~~~~~~~~~\nnmin_{x\in X^\nu} f_0^\nu(x) + h_k^\nu\big(F_k^\nu(x)\big) + \tfrac{1}{2t_k} \|x - x_k^\nu\|_2^2\label{eqn:subprobl}\\
& ~~~~~~~~\text{to produce } x_{k+1}^\nu.\nonumber 
\end{align}
While postponing several details and minor adjustments, the description summarizes the goal of the paper: to find stationary points of the actual problem \eqref{pbm} by solving a sequence of simple master problems \eqref{eqn:subprobl}. In turn, \eqref{eqn:subprobl} leverages gradually better global approximations (increasing $\nu$) and adaptively constructed local models (increasing $k$).

\begin{table}[h]
\centering
\begin{tabular}{c|c|c|c|c|c}
algo. & $X^\nu$ & $f_0^\nu$ & $h^\nu$ & $F^\nu$ & master problem \eqref{eqn:subprobl} \\
\hline
\ref{bundlemethod}  & cl cvx & smooth cvx & cvx Lip           & smooth, Lip grad    & cvx; model of $h$, linearized $F$ \\
\ref{algo-dc}       & cl cvx & smooth cvx & cvx real nondecr  & real dc& cvx; linearize one cvx fcn \\
\ref{algo-dist2set} & cl cvx & smooth cvx & weighted sum      & lLc, dist$^2$       & cvx $+$ nonconvex prj \\
\hline
\end{tabular}
\caption{Assumptions and master problems for three settings corresponding to Algorithms~\ref{bundlemethod}, \ref{algo-dc}, and \ref{algo-dist2set} developed in detail. (cl=closed, cvx=convex, Lip=Lipschitz, real=real-valued, nondecr=nondecreasing, grad=gradient, dc=difference-of-convex, dist$^2$=squared point-to-set distance, fcn=function, prj=projection.)}\label{tab:overview}
\end{table}

As a third contribution, we discuss how double-loop algorithms can be implemented in three settings with different approximating problems, which in turn can come from numerous actual problems. Table~\ref{tab:overview} summarizes the three settings and provides both a road map for the paper and an indication of the breadth of the framework.   

The paper starts in Section~\ref{sec:examples} with three examples. Section~\ref{sec:global} defines stationary points for \eqref{pbm} and \eqref{pbm-a}, and establishes how those of the latter may converge to those of the former. This addresses the outer loop over $\nu$ in the double-loop algorithms. Section~\ref{sec:bundle} focuses on how to compute near-stationary points of an approximating problem \eqref{pbm-a} using local models. It addresses the inner loop over $k$ and results in Algorithm~\ref{bundlemethod}; see Table~\ref{tab:overview}.  Section~\ref{sec:dc} discusses two other approaches for solving the approximating problems \eqref{pbm-a} and produces Algorithms~\ref{algo-dc} and~\ref{algo-dist2set} for the settings summarized in Table~\ref{tab:overview}. An appendix furnishes additional proofs and details.

\medskip

\noindent {\bf Notation and Terminology}. A function $f:\reals^n\to \Reals$ has a {\em domain} $\dom f = \{x\in \reals^n\,|\,f(x)<\infty\}$ and an {\em epigraph} $\epi f = \{(x,\alpha)\in \reals^{n+1}\,|\,f(x)\leq \alpha\}$. It is {\em lower semicontinuous} (lsc) if $\epi f$ is a closed subset of $\reals^n\times\reals$ and it is {\em convex} if $\epi f$ is convex. It is {\em proper} if $\epi f$ is nonempty and $f(x)>-\infty$ for all $x\in \reals^n$. It is {\it locally Lipschitz continuous} (lLc) at $\bar x$ when there are $\delta \in (0,\infty)$ and $\kappa \in [0,\infty)$ such that $|f(x) - f(x')| \leq \kappa\|x-x'\|_2$ whenever $\|x-\bar x\|_2\leq \delta$ and $\|x'-\bar x\|_2\leq \delta$. If $f$ is lLc at every $\bar x\in\reals^n$, then $f$ is lLc. A mapping $F:\reals^n\to \reals^m$ is lLc (at $\bar x$) if its component functions are lLc (at $\bar x$). Functions and mappings are  {\em (twice) smooth} (at $\bar x$) if they are (twice) continuously differentiable (at $\bar x$). The {\em indicator function} of a set $C\subset\reals^n$ is given by $\iota_C(x) = 0$ if $x\in C$ and $\iota_C(x) = \infty$ otherwise. The {\em convex hull} of a set $C$ is denoted by $\con C$. For $C\subset\reals^n$ and $\bar x\in \reals^n$, the {\em point-to-set distance} is denoted by $\dist(\bar x, C)= \inf_{x\in C} \|x-\bar x\|_2$. A function $f:\reals^n\to \Reals$ is {\em nondecreasing} if $f(x)\leq f(x')$ whenever $x\leq x'$, which always is interpreted componentwise.

The collection of subsequences of $\nats=\{1, 2, \dots\}$ is denoted by  $\cN_\infty^\grill$, with convergence of $\{x^\nu, \nu\in\nats\}$ to $x$ along a subsequence $N\in \cN_\infty^\grill$ being denoted by $x^\nu \Nto x$. The symbols $\exists$ and $\forall$ mean ``there exist'' and ``for all,'' respectively. The {\em inner limit} of a sequence of sets $\{C^\nu\subset\reals^n, \nu\in\nats\}$ is $\nInnLim C^\nu = \{x\in \reals^n~|~\exists x^\nu \in C^\nu \to x\}$. The {\em outer limit} is $\nOutLim C^\nu = \{x\in \reals^n~|~\exists N\in \cN_\infty^\grill \mbox{ and } x^\nu \in C^\nu \Nto x\}$. Thus, $C^\nu$ {\em set-converges} to $C\subset\reals^n$, denoted by $C^\nu \sto C$, if $\nOutLim C^\nu \subset C \subset \nInnLim C^\nu$. The functions $\{f^\nu:\reals^n\to \Reals, \nu\in\nats\}$ {\em epi-converge} to $f:\reals^n\to \Reals$, denoted by $f^\nu \eto f$, if $\epi f^\nu \sto \epi f$. 

A set-valued mapping $T:\reals^n\tto \reals^m$ has subsets of $\reals^m$ as its ``values'' and a {\em graph} written as $\gph T = \{(x,y)\in \reals^{n+m}\,|\,y\in T(x)\}$. A sequence of set-valued mappings $\{T^\nu:\reals^n\tto \reals^m, \nu\in\nats\}$ {\em converges graphically} to $T:\reals^n\tto\reals^m$, denoted by $T^\nu \gto T$, if $\gph T^\nu \sto \gph T$. 

The {\em regular normal cone} to $C\subset\reals^n$ at $x\in C$ is denoted by $\widehat N_C(x)$ and consists of those $v\in \reals^n$ satisfying $\langle v, x'-x\rangle \leq o(\|x'-x\|_2)$ for $x'\in C$. The {\em normal cone} to $C$ at $x\in C$ is $N_C(x) = \{v\in\reals^n\,|\,\exists v^\nu\to v, x^\nu\in C\to x \text{ with } v^\nu\in \widehat N_C(x^\nu)\}$. These cones are empty when $x\not\in C$. The set $C$ is {\em Clarke regular} at $x\in C$ if $N_C(x) = \widehat N_C(x)$ and $C$ is closed in a neighborhood of $x$. 

The {\em subdifferential} of $f:\reals^n\to \Reals$ at a point $x$ with $f(x)\in\reals$ is $\partial f(x) = \{v\in\reals^n\,|\,(v,-1)\in N_{\epi f}(x,f(x))\}$ and, when $f$ is lsc, the {\em horizon subdifferential} $\partial^\infty f(x) = \{v\in\reals^n\,|\,(v,0)\in N_{\epi f}(x,f(x))\}$. A convexification is sometimes convenient and we write $\bar\partial f(x) = \con \partial f(x)$. When $f$ is lLc, then $\bar\partial f(x)$ is the {\em Clarke subdifferential} of $f$ at $x$. If $\epi f$ is Clarke regular at $(x,f(x))$, then $f$ is {\em epi-regular} at $x$ and $\partial f(x)$ is convex. If $f$ is convex and $\epsilon\geq 0$, then its $\epsilon$-{\em subdifferential} at $x$ is  $\partial_\epsilon f(x) = \{v\in\reals^n\,|\,f(x') \geq f(x) + \langle v, x'-x\rangle -\epsilon\}$.

\section{Motivating Examples}\label{sec:examples}

The actual problem \eqref{pbm} arises in many contexts including reliability engineering, optimization under uncertainty, and management of energy storage. We give details in three specific cases.

\begin{example}{\rm (buffered failure probability constraints).}\label{ex:buffered}
Let $\psi_{ik}:\reals^{n}\to \reals$, $i=1, \dots, s$, $k=1, \dots, r$, be smooth functions defining the performance of an engineering system. For $x\in\reals^{n}$, $\psi_{ik}(x)$ is the performance of the $k$th component of the system under design (decision) $x$ in scenario $i$. We assume that these functions are ``normalized'' so that positive values of $\psi_{ik}(x)$ represent failure of component $k$ and nonpositive values are deemed acceptable. The overall system performance in scenario $i$ is given by
\[
\psi_i(x) = \max_{j=1, \dots, q} \min_{k \in \mathbb{K}_j} \psi_{ik}(x),
\]
where $\mathbb{K}_j \subset \{1, \dots, r\}$, $j = 1, \dots, q$ is a collection of ``cut sets.'' If all the components in any cut set fail, i.e., $\psi_{ik}(x) > 0$ for all $k\in \mathbb{K}_j$ and some $j$, then the system has failed (in scenario $i$) as indicated by $\psi_i(x) > 0$; see, e.g., \cite[\S\, 6.G]{Royset_Wets_2021} and \cite{ByundeOliveiraRoyset.23}.

We assume that the probability of scenario $i$ is $p_i>0$. Given $\alpha \in (0,1)$, we seek to minimize a smooth function $f_0:\reals^{n}\to \reals$ over a nonempty closed set $X\subset \reals^{n}$ subject to the buffered failure probability of the system being no greater than $1-\alpha$; consult \cite{RockafellarRoyset.10,Royset2025}, and \cite[\S\, 3.E]{Royset_Wets_2021} for background about such formulations. These sources specify that the buffered failure probability constraint is equivalently stated in terms of a superquantile/CVaR constraint, which in turn results in the constraint 
\[
\max_{\pi \in P_\alpha}  \sum_{i = 1}^s \pi_i \psi_i(x)  \leq 0,
\quad \mbox{with} \quad
P_\alpha=\bigg\{\pi\in \Re^s~\bigg|~\sum_{i=1}^s \pi_i=1,~0\leq \pi_i\leq\frac{ p_i}{1-\alpha}\bigg\}.
\] 
The problem of minimizing $f_0$ over $X$ subject to the buffered failure probability constraint therefore becomes
\[
\nnmin_{x\in X} ~ f_0(x) ~\mbox{ subject to } ~ \max_{\pi \in P_\alpha}  \sum_{i = 1}^s \pi_i \max_{j=1, \dots, q} \min_{k \in \mathbb{K}_j} \psi_{ik}(x)  \leq 0,
\]
which is of the form \eqref{pbm} with
\[
h(u) = \iota_{(-\infty,0]}\bigg(\max_{\pi \in P_\alpha}  \sum_{i = 1}^s \pi_i \max_{j=1, \dots, q} u_{ij}\bigg), ~~ u = ( u_{11}, \dots, u_{sq}),
\]
defining a proper, lsc, and convex function, and the lLc mapping $F:\reals^{n} \to \reals^{sq}$ given by $F(x) = (\min_{k \in \mathbb{K}_1} \psi_{1k}(x), \dots, \min_{k \in \mathbb{K}_q} \psi_{sk}(x))$. Approximations $h^\nu$ and $F^\nu$ may take the form 
\[
h^\nu(u)= \rho^\nu\max\bigg\{0,\;\max_{\pi \in P_\alpha}  \sum_{i = 1}^s \pi_i \max_{j=1, \dots, q} u_{ij}\bigg\}\;\mbox{ and }\; 
F^\nu(x)=\big(\texttt{LSE}^\nu_{11}(x), \dots, \texttt{LSE}^\nu_{sq}(x)\big),
\]
where $\rho^\nu>0$ is a penalty parameter and 
\[
\mathtt{LSE}^\nu_{ij}(x)=-\frac{\eta^\nu}{\ln |\mathbb{K}_j|}\ln\Bigg(\sum_{k\in \mathbb{K}_j} \exp\Big(-\frac{\ln |\mathbb{K}_j|}{\eta^\nu} \psi_{ik}(x)\Big)\Bigg)
\]
is the standard \emph{LogSumExp} function (see, e.g., \cite[Examples 4.16, 6.43, and 7.43]{Royset_Wets_2021}) parameterized by $\eta^\nu>0$. Here, $|\mathbb{K}_j|$ is the cardinality of $\mathbb{K}_j$.These global approximations furnish the approximating problems~\eqref{pbm-a}, whose structure is compatible with local models to complete the construction of a double-loop algorithm.\mybox
\end{example}

\begin{example}{\rm (distributionally robust two-stage nonconvex programs).}\label{ex:ncvxSP} For a nonempty closed set $X \subset \Re^n$, a smooth function $f_0:\reals^n\to \reals$ representing the first-stage cost, and a compact set $\mathcal{P}$ of probability vectors $p = (p_1, \dots, p_s)$, we consider the problem
\[
\nnmin_{x\in X}  f_0(x) + \max_{p\in \mathcal{P}} \displaystyle \sum_{i=1}^s p_i f_i(x), ~~\text{where} ~~
f_i(x) = \inf_{y \in Y_i}\big\{ \phi_i(x,y)~\big|~ \psi_{i}(x,y)\leq 0 \big\}
\]
is defined in terms of a nonempty, convex, and compact set $Y_i \subset \Re^{m}$, a second-stage cost function $\phi_i:\Re^n\times \Re^m \to \Re$, assumed to be smooth, and a second-stage constraint mapping $\psi_{i}:\Re^n\times \Re^m \to \Re^{q}$, also assumed to be smooth. If $\phi_i(x,\cdot)$ and $\psi_{i}(x,\cdot)$ are convex for each $x$ and a qualification holds, then $f_i$ is lLc; see Example~\ref{ex:ncvxSP2} for details. 
The problem is then of the form \eqref{pbm} with $h(u)=\max_{p\in \mathcal{P}} \sum_{i=1}^s p_i u_i$,  $u=(u_1,\ldots,u_s)$, which is real-valued and convex, and $F:\Re^n \to \Re^{s}$ given by $F(x) = (f_1(x),\ldots,f_s(x))$ is lLc but possibly nonsmooth.  
With $\eta^\nu>0$, we achieve a smooth approximation of $F$ under additional assumptions by setting 
\[
f^\nu_i(x)= \inf_{y \in Y_i} \big\{ \phi_i(x,y) + \tfrac{1}{2}\eta^\nu\norm{y}^2~\big|~\psi_{i}(x,y)\leq 0\big\}
\]
and $F^\nu(x)=(f_1^\nu(x),\ldots,f_s^\nu(x))$; see Example~\ref{ex:ncvxSP2}. These global approximations can then be paired with local models.\mybox
\end{example}

\begin{example}{\rm (minimization under distance-to-set penalties).}\label{example:dist2set}
For a nonempty closed set $X\subset \Re^n$ and smooth functions $f_i\fc{\Re^n}{\Re}$, $i=0, 1, \dots, m$, we seek to minimize $f_0$ over the ``hard constraints'' $x\in X, f_1(x) \leq 0, \dots, f_m(x) \leq 0$, while accounting for the preference that $x$ should be in the closed sets $K_1, \ldots, K_{q}\subset\reals^n$. While these sets are abstract and possibly nonconvex, we assume that the projection on $K_i$ is implementable in practice, i.e., there is an efficient algorithm for computing $x\in \proj_{K_i}(\bar x) = \argmin_{x'\in K_i} \|x'-\bar x\|_2$ for $\bar x\in \reals^n$. Interpreting the requirements about $K_1, \dots, K_q$ as ``soft constraints'' and using penalty coefficients $\rho_1, \dots, \rho_q$, we obtain 
\[
\nnmin_{x \in X} \; f_0(x)+ \sum_{i=1}^{q} \frac{\rho_i}{2} \dist^2(x,K_i)\quad \mbox{subject to}\quad f_i(x) \leq 0,   ~~i=1, \dots, m,
\]
where $\dist^2(x,K_i) = (\dist(x,K_i))^2$. A mixture of hard and soft constraints are common in models of energy systems such as those related to battery energy storage systems, where battery health and charge-discharge cycles may enter as soft constraints $x \in K_i$ \cite{Dhekra_Oliveira_Pflaum_2020}. Demand satisfaction, physical restrictions, and contractual requirements tend to enter as hard constraints. The formulation is of the form~\eqref{pbm} with
\[
h(z)= \sum_{i=1}^{q} \frac{\rho_i}{2} z_i + \sum_{i=1}^m \iota_{(-\infty,0]}(z_{q+i})\;\mbox{ and }\;
F(x)=\big(\dist^2(x,K_1),\ldots,\dist^2(x,K_{q}), f_1(x), \dots, f_m(x)\big).
\]
Since $\dist^2(\cdot,K_i)$ can be written as a real-valued dc function \cite[Proposition 4.13]{WWBook} (see~\eqref{eq:dist2set-dc} below), it is lLc and $F$ is therefore lLc. An approximation of $h$ might be given by $h^\nu(z)= \sum_{i=1}^{q} \frac{\rho_i}{2} z_i + \sum_{i=1}^m \eta_i^\nu \max\{0, z_{q+i})$ for positive penalty parameters $\eta_i^\nu$. Again, these global approximations can be supplemented with local models.\mybox
\end{example}

\section{Optimality Conditions and Global Approximations}\label{sec:global}

This section establishes that cluster points of near-stationary points of the approximating problems \eqref{pbm-a} are stationary points of the actual problem \eqref{pbm} under mild assumptions. The results furnish a justification for the outer portion of double-loop algorithms, but are also of independent interest. 

Throughout, we retain flexibility by defining stationarity for the actual problem \eqref{pbm} and the approximating problems \eqref{pbm-a} in terms of the set-valued mappings $S,S^\nu:\reals^n\times\reals^m\times\reals^m \tto  \reals^m \times \reals^m\times \reals^n$ expressed as:  
\begin{align}
S(x,y,z) & = \big\{F(x) - z\big\}  \times \big( \partial h(z)- y\big) \times  \Big( \nabla f_0(x) + \sum_{i=1}^m y_i \con D_i(x) + N_X(x)\Big)\label{eqn:S}\\
S^\nu(x,y,z) & = \big\{F^\nu(x) - z\big\}  \times \big( \partial h^\nu(z)- y\big) \times  \Big( \nabla f_0^\nu(x) + \sum_{i=1}^m y_i \con D_i^\nu(x) + N_{X^\nu}(x)\Big),\label{eqn:Snu}
\end{align}
where $f_1, \dots, f_m$ and $f_1^\nu, \dots, f_m^\nu$ are the component functions of $F$ and $F^\nu$, respectively. The set-valued mappings $D_i,D^\nu_i:\reals^n\tto\reals^n$ capture first-order differential properties of $f_i$ and $f_i^\nu$, respectively, but are left unspecified to allow for a multitude of possibilities. For example, if these functions are smooth, then the most natural choice would be $D_i(x) = \{\nabla f_i(x)\}$ and $D_i^\nu(x) = \{\nabla f_i^\nu(x)\}$. If $f_i$ is merely lLc, then $D_i(x) = \partial f_i(x)$ is reasonable. Another possibility arises when $f_i = f_i^1 - f_i^2$ for real-valued convex functions $f_i^1$ and $f_i^2$, i.e., $f_i$ is a dc function. Then, $D_i(x) = \partial f_i^1(x) - \partial f_i^2(x)$ or $D_i(x) = \partial f_i^1(x) - \partial_\epsilon f_i^2(x)$ might be suitable.     

We say that the {\em basic qualification} holds at $x$ if: 
\begin{equation}\label{qualification}
y\in N_{\dom h}\big(F(x)\big) ~~ \mbox{ and }~ ~  0 \in \sum_{i=1}^m y_i \con D_i(x) + N_X(x)~~~\Longrightarrow~~~ y=0.
\end{equation}

The basic qualification holds, for example, when $h$ is real-valued and reduces to the ``transversality condition'' in \cite[equation (4.3)]{LewisWright.16} when $X=\reals^n$, $f_1, \dots, f_m$ are smooth, and $D_1 = \{\nabla f_1\}, \dots, D_m = \{\nabla f_m\}$; see also \cite[Example 10.8]{Rockafellar_Wets_1998}. 

The following necessary condition extends \cite{Royset2023a} by allowing for general set-valued mappings $D_1, \dots, D_m$ and also including a function $f_0$; we provide a proof in the appendix. 

\begin{proposition}{\rm (optimality condition).}\label{prop:optimality} For a nonempty, closed set $X\subset\reals^n$, a proper, lsc, and convex function $h:\reals^m\to \Reals$, a smooth function $f_0:\reals^n\to \reals$, and a lLc mapping $F:\reals^n\to \reals^m$, with component functions $f_1, \dots, f_m$, suppose that the basic qualification \eqref{qualification} holds at $x^\star$ and the set-valued mappings $D_1, \dots, D_m:\reals^n\tto \reals^n$ satisfy $\partial f_i(x^\star) \subset D_i(x^\star)$ for $i = 1, \dots, m$.

If $x^\star$ is a local minimizer of $\iota_X + f_0 + h \circ F$, then there exist $y^\star\in\reals^m$ and $z^\star\in \reals^m$ such that 
\[
0 \in S(x^\star, y^\star, z^\star),
\]
with this optimality condition being equivalent to the Fermat rule 
\[
0 \in \partial(\iota_X + f_0 + h \circ F)(x^\star)
\]
under the additional assumptions that $X$ is Clarke regular at $x^\star$ and, for each $y\in \partial h(F(x^\star))$ and $i = 1, \dots, m$, the function $y_i f_i$ is epi-regular at $x^\star$ and $D_i(x^\star) = \partial f_i(x^\star)$.
\end{proposition}

We say that $x$ is a {\em stationary point} for \eqref{pbm} if there exist $y\in\reals^m$ and $z\in\reals^m$ such that $0 \in S(x,y,z)$. Parallel optimality conditions expressed by $S^\nu$ follow for the approximating problems \eqref{pbm-a}. The vector $z^\star$ in Proposition~\ref{prop:optimality} appears superfluous; one can simply set $z^\star = F(x^\star)$. However, it plays an important role in the presence of approximations and tolerances. We recall that having a subgradient near the origin is not a suitable stopping criterion because it may never be satisfied for an algorithm applied to, for example, the absolute value function $|\cdot|$. Nevertheless, the absolute value function fits our framework with $X = \reals$, $f_0(x) = 0$, $h = |\cdot|$, $F(x) = f_1(x) = x$, and $D_1(x) = \{\nabla f_1(x)\}$. Then, $\dist(0,S(x,y,z))\leq \epsilon$ for every $x\in [-\epsilon, \epsilon]$ because one can set $y = z = 0$.   

The main result of this section provides a justification for the following conceptual algorithm, stated as Algorithm \ref{algo}, representing the outer loop of the double-loop algorithms for \eqref{pbm}.

\begin{algorithm}[htb]
\caption{Global Approximation Algorithm}
\label{algo}
\begin{algorithmic}[1]
{\footnotesize
\State Data: $\{\epsilon^\nu > 0, X^\nu\subset\reals^n, f_0^\nu:\reals^n\to \reals, h^\nu:\reals^m\to \Reals, F^\nu:\reals^n\to \reals^m, D_i^\nu:\reals^n\tto\reals^n, i=1, \dots, m, \nu\in\nats\}$
\For{$\nu = 1, 2, \dots$}
\State Compute $x^\nu$ such that
\begin{equation}\label{eqn:tol}
\dist\big(0, S^\nu(x^\nu,y^\nu,z^\nu)\big) \leq \epsilon^\nu
\end{equation}
\quad \;\;for some $y^\nu,z^\nu \in \reals^m$
\EndFor
}
  \end{algorithmic}
\end{algorithm}

Instead of prescribed tolerances $\{\epsilon^\nu, \nu\in\nats\}$, Algorithm \ref{algo} can be implemented with other stopping rules as long as $\dist(0, S^\nu(x^\nu,y^\nu,z^\nu))$ vanishes as we see in Sections \ref{sec:bundle} and \ref{sec:dc}. This allows us to merely compute $x^\nu$ in Line 3 and bypass the computation of $y^\nu,z^\nu$ if desirable. Below we describe Algorithms~\ref{bundlemethod}, \ref{algo-dc}, and \ref{algo-dist2set} as concrete approaches of how to implement Line 3. However, the following convergence analysis is agnostic about how Line 3 is accomplished.

\begin{theorem}\label{thm:optim}{\rm (convergence of global approximation algorithm)}. 
For nonempty, closed, and convex sets $X,X^\nu\subset\reals^n$, proper, lsc, and convex functions $h,h^\nu:\reals^m\to \Reals$, smooth functions $f_0,f_0^\nu:\reals^n\to \reals$, lLc mappings $F,F^\nu:\reals^n\to \reals^m$, with component functions $f_1, \dots, f_m$ and $f_1^\nu, \dots, f_m^\nu$, respectively, and set-valued mappings $D_i,D_i^\nu:\reals^n\tto\reals^n$, $i = 1, \dots, m$, suppose that 
\begin{enumerate}[label=(\alph*), start=1]
\item $X^\nu\sto X$

\item $h^\nu\eto h$

\item $f_0^\nu(x^\nu)\to f_0(x)$ and $\nabla f_0^\nu(x^\nu) \to \nabla f_0(x)$ whenever $x^\nu\in X^\nu \to x$

\item for $i = 1, \dots, m$, one has the property: 
\begin{equation}\label{eqn:suffFconv}
\hspace{-0.15cm}  \begin{cases}
    x^\nu \in X^\nu \to x\\
    v^\nu \in D_i^\nu(x^\nu)
  \end{cases}
  \hspace{-0.35cm}\Longrightarrow 
  \begin{cases}
    f_i^\nu(x^\nu) \to f_i(x)\\
    \{v^\nu, \nu\in\nats\} \mbox{ is bounded with all its cluster points in } \con D_i(x).
  \end{cases}
\end{equation}
\end{enumerate}
Then, the following hold:
\begin{enumerate}[label=(\roman*), start=1]

\item  Suppose that $\{(x^\nu,y^\nu,z^\nu), \nu\in\nats\}$ satisfies \eqref{eqn:tol}, with $\{\epsilon^\nu \in [0,\infty), \nu\in\nats\} \to 0$, and $\{x^\nu, \nu\in N\}$ converges to a point $\bar x$ along a subsequence $N\in \cN_\infty^\grill$. Let the basic qualification \eqref{qualification} hold at $\bar x$ if $F(\bar x) \in \dom h$ and let the following qualification hold if $F(\bar x) \not\in \dom h$: 
\begin{equation}\label{qualificationInfeas}
0 \in \sum_{i=1}^m y_i \con D_i(\bar x) + N_X(\bar x)~~~\Longrightarrow~~~ y=0.
\end{equation}
Then, there exist $\bar y\in\reals^m$ and $\bar z\in \reals^n$ such that 
\[
0 \in S(\bar x, \bar y, \bar z).
\]

\item If $\{(x^\nu,y^\nu,z^\nu), \nu\in\nats\}$ satisfies \eqref{eqn:tol}, with $\{\epsilon^\nu \in [0,\infty), \nu\in\nats\} \to 0$, and it has a cluster point $(\bar x, \bar y, \bar z)$, then $0 \in S(\bar x, \bar y, \bar z)$.

\item  If all $F,F^\nu$ are smooth and all $D_i = \{\nabla f_i\}$ and $D^\nu_i= \{\nabla f_i^\nu\}$, then one has: 
\begin{equation}\label{eqn:liminnconseq}
\hspace{-0.2cm}\forall  (\bar x, \bar y, \bar z) \in S^{-1}(0), \exists (x^\nu,y^\nu,z^\nu) \to (\bar x, \bar y, \bar z) \mbox{ and } \epsilon^\nu\to 0 \mbox{ with } \dist\big(0, S^\nu(x^\nu,y^\nu,z^\nu)\big) \leq \epsilon^\nu. 
\end{equation}

\item If $X^\nu = X$, $F^\nu = F$, and $D_i^\nu = D_i$ for all $\nu\in\nats$ and $i = 1, \dots, m$, then \eqref{eqn:liminnconseq} holds.  
\end{enumerate}
If $X^\nu = X$ for all $\nu\in \nats$, then (i)-(iv) hold without the convexity assumption on $X$.   
 
\end{theorem}  
 \begin{proof} For (ii), we start by establishing the property $\nOutLim (\gph S^\nu) \subset \gph S$. Let $(\bar x, \bar y, \bar z, \bar u, \bar v, \bar w)$ $\in$ $\nOutLim (\gph S^\nu)$.  Then, there are $N\in \cN_\infty^\grill$, $\bar x^\nu\Nto \bar x$, $\bar y^\nu\Nto \bar y$, $\bar z^\nu\Nto \bar z$, $\bar u^\nu\Nto \bar u$, $\bar v^\nu\Nto \bar v$, and $\bar w^\nu\Nto \bar w$ with
$(\bar x^\nu, \bar y^\nu, \bar z^\nu, \bar u^\nu, \bar v^\nu, \bar w^\nu)$ $\in$ $\gph S^\nu$. Consequently,
\[
\bar u^\nu = F^\nu(\bar x^\nu) - \bar z^\nu, ~~~ \bar v^\nu \in \partial h^\nu(\bar z^\nu)- \bar y^\nu, ~~~ \bar w^\nu \in \nabla f_0^\nu(\bar x^\nu) + \sum_{i=1}^m \bar y_i^\nu \con D_i^\nu(\bar x^\nu) + N_{X^\nu}(\bar x^\nu).
\]
The existence of $\bar w^\nu$ implies that $\bar x^\nu\in X^\nu$ because otherwise $N_{X^\nu}(\bar x^\nu)$ would have been an empty set. Since $X^\nu\sto X$, we also have $\bar x\in X$. Thus,  $F^\nu(\bar x^\nu)\Nto F(\bar x)$ and we conclude that $\bar u = F(\bar x) - \bar z$. Attouch's theorem \cite[Theorem 12.35]{Rockafellar_Wets_1998} states that $h^\nu \eto h$ implies $\partial h^\nu \gto \partial h$. Consequently, $\bar v \in \partial h(\bar z) - \bar y$.

The inclusion for $w^\nu$ implies that there exist
\begin{equation}\label{eqn:constProofw}
a_i^\nu \in \con D_i^\nu(\bar x^\nu), ~i=1, \dots, m,~ ~~\mbox{ such that } ~~~ \bar w^\nu - \nabla f_0^\nu(\bar x^\nu) - \sum_{i=1}^m \bar y_i^\nu a_i^\nu \in N_{X^\nu}(\bar x^\nu).
\end{equation}
By Caratheodory's theorem, there are $\lambda_{ij}^\nu \geq 0$ and $b_{ij}^\nu\in D_i^\nu(\bar x^\nu)$, $j=0, 1, \dots, n$, such that $\sum_{j=0}^n \lambda_{ij}^\nu = 1$ and $a_i^\nu = \sum_{j=0}^n \lambda_{ij}^\nu b_{ij}^\nu$. The sequence $\{(\lambda_{i0}^\nu, \dots, \lambda_{in}^\nu), \nu\in N\}$ is contained in a compact set and thus has a convergent subsequence. By assumption \eqref{eqn:suffFconv}, $\{b_{ij}^\nu, \nu \in N\}$ is bounded with all its cluster points in $\con D_i(\bar x)$. Consequently, there exist $\bar \lambda_i = (\bar \lambda_{i0}, \dots, \bar \lambda_{in})\in \reals^{1+n}$, $\bar b_{i0}, \dots, \bar b_{in}$, and a subsequence of $N$, which we also denote by $N$, such that
\[
(\lambda_{i0}^\nu, \dots, \lambda_{in}^\nu) \Nto \bar\lambda_i \geq 0, ~\mbox{ with } ~\sum_{j=0}^n \bar \lambda_{ij} = 1,  ~~~\mbox{ and } ~~~ b_{ij}^\nu \Nto \bar b_{ij} \in \con D_i(\bar x).
\]
This means that $a_i^\nu \Nto \bar a_i = \sum_{j=0}^n \bar \lambda_{ij} \bar b_{ij} \in \con D_i(\bar x)$. Again, via Attouch's theorem, $\ind_{X^\nu} \eto \ind_X$ implies $N_{X^\nu} \gto N_X$ because $X^\nu$ and $X$ are convex. Assumption (c) ensures that $\nabla f_0^\nu(\bar x^\nu) \Nto \nabla f_0(\bar x)$. Using these facts as well as the observations that $\bar y_i^\nu \Nto \bar y_i$, result in  
\[
\bar w^\nu - \nabla f_0^\nu(\bar x^\nu) - \sum_{i=1}^m \bar y_i^\nu a_i^\nu ~\Nto~ \bar w  - \nabla f_0(\bar x) - \sum_{i=1}^m \bar y_i \bar a_i.
\]
The inclusion for $\bar w^\nu$ in \eqref{eqn:constProofw} implies that
\[
\bar w - \nabla f_0(\bar x) - \sum_{i=1}^m \bar y_i \bar a_i \in N_{X}(\bar x).
\]
Thus, because $\bar a_i \in \con D_i(\bar x)$, one has
\[
\bar w \in \nabla f_0(\bar x) + \sum_{i=1}^m \bar y_i \con D_i(\bar x) + N_{X}(\bar x)
\]
and we conclude that $(\bar u, \bar v, \bar w) \in S(\bar x, \bar y, \bar z)$. We have established that $\nOutLim (\gph S^\nu) \subset \gph S$.

Under the modified assumption that $X^\nu = X$, but not necessarily convex, we still have $N_{X^\nu} =N_X \gto N_X$ because $N_X$ is outer semicontinuous; see \cite[Proposition 6.6]{Rockafellar_Wets_1998}. Thus, the above argument carries over.

In either case, $\nOutLim (\gph S^\nu) \subset \gph S$ implies via \cite[Convergence statement 7.31]{Royset_Wets_2021} the assertion in (ii). 

\medskip

For (i), since $\epsilon^\nu$ is never infinity and $\epsilon^\nu\to 0$, there is vanishing $\{w^\nu\in\reals^n, \nu\in\nats\}$ such that 
\begin{align}\label{eqn:winclusion}
w^\nu \in \nabla f_0^\nu(x^\nu) + \sum_{i=1}^m y_i^\nu \con D_i^\nu(x^\nu) + N_{X^\nu}(x^\nu)~~~\forall \nu\in\nats.
\end{align}
Thus, $x^\nu\in X^\nu$ for all $\nu$ because $N_{X^\nu}(x^\nu)$ would otherwise have been an empty set. This fact together with \eqref{eqn:suffFconv} imply that $F^\nu(x^\nu)\Nto F(\bar x)$. Moreover, $F^\nu(x^\nu) - z^\nu\to 0$  because $\epsilon^\nu \to 0$. In combination, these results yield $z^\nu \Nto F(\bar x)$. Set $\bar z = F(\bar x)$. We examine two cases; the second is shown to be impossible.

\medskip

\noindent {\bf Case 1}. Suppose that there exists a subsequence $N' \subset N$ such that $\{y^\nu, \nu\in N'\}$ is bounded. Then, there exist a further subsequence $N''\subset N'$ and a point $\bar y$ such that $y^\nu$ converges to $\bar y$ along $N''$, i.e., $(\bar x, \bar y, \bar z)$ is a cluster point of $\{(x^\nu,y^\nu,z^\nu), \nu\in \nats\}$. It follows from (ii) that $0 \in S(\bar x, \bar y, \bar z)$.

\medskip
  
\noindent {\bf Case 2}. Suppose that $\|y^\nu\|_2 \Nto \infty$. Since $y^\nu \neq 0$ for all $\nu \in N$ sufficiently large, $v^\nu = y^\nu/\|y^\nu\|_2$ is well defined for such $\nu$. The boundedness of $\{v^\nu, \nu\in N\}$ implies the existence of $\bar v \in \reals^m$ and a subsequence $N'\subset N$ such that $\{v^\nu, \nu\in N'\}$ is well defined and $v^\nu \to \bar v$ along $N'$.  

The inclusion for $w^\nu$ in \eqref{eqn:winclusion} implies that there exist $a_i^\nu \in \con D_i^\nu(x^\nu)$, $i=1, \dots, m$, such that 
\begin{equation}\label{eqn:constProofw2}
w^\nu - \nabla f_0^\nu(x^\nu) - \sum_{i=1}^m y_i^\nu a_i^\nu \in N_{X^\nu}(x^\nu).
\end{equation}
Following an argument similar to that after \eqref{eqn:constProofw}, we find that there exist points $\bar a_i \in \con D_i(\bar x), i=1, \dots, m$, and a subsequence $N''\subset N'$ along which $a_i^\nu$ converges to $\bar a_i$. From \eqref{eqn:constProofw2} and using the fact that the inclusion involves a cone, we obtain that  
\[
w^\nu/\|y^\nu\|_2 - \nabla f_0^\nu(x^\nu)/\|y^\nu\|_2 - \sum_{i=1}^m v_i^\nu a_i^\nu \in N_{X^\nu}(x^\nu) ~~~\forall \nu\in N''.
\]
Via Attouch's theorem, $\ind_{X^\nu} \eto \ind_X$ implies $N_{X^\nu} \gto N_X$ because $X^\nu$ and $X$ are convex. In the absence of convexity but under the additional assumption that $X^\nu = X$ for all $\nu$, the same graphical convergence holds because $N_X$ is outer semicontinuous; see, e.g., \cite[Proposition 7.22]{Royset_Wets_2021}. Using this fact as well as the observation that  
\[
w^\nu/\|y^\nu\|_2 - \nabla f_0^\nu(x^\nu)/\|y^\nu\|_2 - \sum_{i=1}^m v_i^\nu a_i^\nu ~\mbox{ converges to } - \sum_{i=1}^m \bar v_i \bar a_i \mbox{ along $N''$}
\]
because $\nabla f_0^\nu(x^\nu)\Nto \nabla f_0(\bar x)$, we obtain that $- \sum_{i=1}^m \bar v_i \bar a_i \in N_{X}(\bar x)$. Thus, because $\bar a_i \in \con D_i(\bar x)$, one has
\begin{equation}\label{eqn:contraQproof}
0 \in \sum_{i=1}^m \bar v_i \con D_i(\bar x) + N_{X}(\bar x).
\end{equation}
Since $\bar v \neq 0$, we have contradicted the qualification \eqref{qualificationInfeas} if $F(\bar x) \not\in \dom h$. It only remains to consider the possibility $F(\bar x) \in \dom h$. 

Suppose that $\bar z = F(\bar x) \in \dom h$. There exists a sequence $\{c^\nu \in \reals^m, \nu\in\nats\}$ with $c^\nu \to 0$ such that $y^\nu + c^\nu \in \partial h^\nu(z^\nu)$ for all $\nu\in\nats$ because $\epsilon^\nu\to 0$. The subgradient inequality for convex functions establishes that 
\[
h^\nu(z') \geq h^\nu(z^\nu) + \langle y^\nu + c^\nu, z' - z^\nu\rangle ~~~\forall z' \in \reals^m.
\]
For $\nu\in N'$, one has
\begin{equation}\label{eqn:cvxineq}
\big(h^\nu(z') - h^\nu(z^\nu)\big)/\|y^\nu\|_2\geq   \big\langle v^\nu + c^\nu/\|y^\nu\|_2, z' - z^\nu\big\rangle ~~~\forall z' \in \reals^m.
\end{equation}
Let $\hat z \in \dom h$. Since $h^\nu\eto h$, $h$ is proper, and $\bar z\in \dom h$, a characterization of epi-convergence (see, e.g., \cite[Theorem 4.15]{Royset_Wets_2021}) ensures that $\nliminf_{\nu\in N} h^\nu(z^\nu) \geq h(\bar z) \in \reals$ and, for some $\hat z^\nu\to \hat z$, $h^\nu(\hat z^\nu) \to h(\hat z)\in\reals$. The quantity $h^\nu(\hat z^\nu) - h^\nu(z^\nu)$ is therefore bounded from above as $\nu\to \infty$ along $N$. Thus, \eqref{eqn:cvxineq} with $z' = \hat z^\nu$, implies after passing to the limit (along $N'$) that $0 \geq \langle \bar v, \hat z - \bar z\rangle$. Since $\hat z \in \dom h$ is arbitrary, it follows that $\bar v \in N_{\dom h}(\bar z)$. In combination with \eqref{eqn:contraQproof}, we have contradicted the basic qualification \eqref{qualification}. Thus, Case 2 is not possible.

\medskip

For (iii) and (iv), it suffices to show $\nInnLim (\gph S^\nu) \supset \gph S$ due to \cite[Theorem 5.37(b)]{Rockafellar_Wets_1998}. First, consider (iii). Now, we obtain the simplification
\[
\sum_{i=1}^m y_i \con D_i(x) = \Big\{\sum_{i=1}^m y_i \nabla f_i(x)\Big\} = \big\{\nabla F(x)^\top y\big\},
\]
with a similar expression for the approximating functions. Let $(\bar x, \bar y, \bar z, \bar u, \bar v, \bar w) \in \gph S$. This means that $\bar u = F(\bar x) - \bar z$, $\bar v \in \partial h(\bar z)- \bar y$, and $\bar w \in \nabla f_0(\bar x) + \nabla F(\bar x)^\top \bar y + N_{X}(\bar x)$.
In particular, $\bar x\in X$. Set $\bar y^\nu = \bar y$. Since $h^\nu\eto h$, it follows by Attouch's theorem that $\partial h^\nu \gto \partial h$ and there are $\bar v^\nu \to \bar v$ and $\bar z^\nu\to \bar z$ such that $\bar v^\nu + \bar y \in \partial h^\nu(\bar z^\nu)$.

Since $N_{X^\nu} \gto N_X$, there's $(\bar x^\nu,t^\nu) \in \gph N_{X^\nu}$ with $\bar x^\nu \in X^\nu \to \bar x$ and $t^\nu \to \bar w - \nabla f_0(\bar x) - \nabla F(\bar x)^\top \bar y$. Construct $\bar w^\nu = t^\nu +  \nabla f_0^\nu(\bar x^\nu) + \nabla F^\nu(\bar x^\nu)^\top \bar y$. We then have $\bar w^\nu \to \bar w$ and $\bar w^\nu - \nabla f_0^\nu(\bar x^\nu) - \nabla F^\nu(\bar x^\nu)^\top \bar y \in N_{X^\nu}(\bar x^\nu)$. Also, construct $\bar u^\nu = F^\nu(\bar x^\nu) - \bar z^\nu$, which converges to $\bar u$.
In summary, we have constructed 
\[
(\bar x^\nu, \bar y^\nu, \bar z^\nu, \bar u^\nu, \bar v^\nu, \bar w^\nu) \in \gph S^\nu \to (\bar x, \bar y, \bar z, \bar u, \bar v, \bar w).
\]
This means that $(\bar x, \bar y, \bar z, \bar u, \bar v, \bar w) \in \nInnLim (\gph S^\nu)$ and, thus, $\nInnLim (\gph S^\nu) \supset \gph S$ holds.

Again the convexity of $X,X^\nu$ is not needed if these sets coincide. 

Second, consider (iv). Let $(\bar x, \bar y, \bar z, \bar u,
\bar v, \bar w) \in \gph S$, i.e., 
\[
\bar u = F(\bar x) - \bar z, ~~~ \bar v \in \partial h(\bar z)- \bar y, ~~~ \bar w \in \nabla f_0(\bar x) + \sum_{i=1}^m \bar y_i \con D_i(\bar x) + N_{X}(\bar x).
\]
Set $\bar x^\nu = \bar x$, $\bar y^\nu = \bar y$, and $\bar w^\nu = \bar w + \nabla f_0^\nu(\bar x) - \nabla f_0(\bar x)$. By assumption (c), $\bar w^\nu\to \bar w$. Since $h^\nu\eto h$, it follows by Attouch's theorem that $\partial h^\nu \gto \partial h$ and there are $\bar v^\nu \to \bar v$ and $\bar z^\nu\to \bar z$ such that $\bar v^\nu + \bar y \in \partial h^\nu(\bar z^\nu)$. Let $\bar u^\nu = F(\bar x) - \bar z^\nu$. This means that $(\bar x^\nu, \bar y^\nu, \bar z^\nu, \bar u^\nu, \bar v^\nu, \bar w^\nu) \in \gph S^\nu \to (\bar x, \bar y, \bar z, \bar u, \bar v, \bar w)$. Consequently, $(\bar x, \bar y, \bar z, \bar u, \bar v, \bar w) \in \nInnLim (\gph S^\nu)$ and $\nInnLim (\gph S^\nu) \supset \gph S$ holds.\qed
\end{proof}

The assumptions of the theorem are satisfied in many typical situations. For example, the global approximation $X^\nu = [-\rho^\nu,\rho^\nu]^n$ set-converges to $X = \reals^n$ as $\rho^\nu\to \infty$. Likewise, $h^\nu(z) = \theta^\nu\sum_{i=1}^m \max\{0,z_i\}$ defines a function that epi-converges to $h = \iota_{(-\infty,0]^m}$ as $\theta^\nu\to \infty$. The function $\mathtt{LSE}^\nu_{ij}$ is a global approximation of $f_{ij} = \min_{k\in \mathbb{K}_j} \psi_{ik}$ in Example \ref{ex:buffered} and satisfies the requirements in \eqref{eqn:suffFconv} when $\eta^\nu\to 0$. We refer to \cite{Royset2023a} for additional examples. 

The main novelty of Theorem~\ref{thm:optim} lies in part (i) with parts (ii-iv) representing extensions of results implicit in \cite{Royset2023a} by allowing for arbitrary set-valued mappings $D_i^\nu, D_i$ and also including a smooth function $f_0$. The importance of (i) over (ii) is that the former identifies the basic qualification \eqref{qualification} and the additional qualification \eqref{qualificationInfeas} as sufficient to avoid unbounded sequences $\{(y^\nu,z^\nu), \nu\in\nats\}$, which would have make (ii) inapplicable.  

The additional qualification \eqref{qualificationInfeas} cannot be removed. For $X = X^\nu = \reals$, $h(z) = \ind_{(-\infty,0]}(z)$, $h^\nu(z) = \theta^\nu \max\{0, z\}$, $f_0(x) = f_0^\nu(x) = 0$ for all $x$, $F^\nu(x) = F(x) = \sin x$, and $D_i^\nu(x) = D_i(x) = \{\cos x\}$, set $(x^\nu, y^\nu, z^\nu) = (\pi/2, \theta^\nu, 1)$ for all $\nu\in\nats$. Then, 
\[
S^\nu(x^\nu,y^\nu,z^\nu) = \{\sin x^\nu - z^\nu\} \times \{\theta^\nu - y^\nu\} \times \{y^\nu \cos x^\nu\}  = \{(0,0,0)\}
\]
and $h^\nu \eto h$ provided that $\theta^\nu\to \infty$. Nevertheless, there are no $\bar y\in\reals$ and $\bar z\in \reals$ such that   
\[
0 \in S(\bar x,\bar y, \bar z) = \{1 - \bar z\} \times \big(N_{(-\infty,0]}(\bar z) - \bar y\big) \times \{0\}  
\]
because $N_{(-\infty,0]}(1) = \emptyset$. In fact while satisfying the optimality condition for each approximating problem, $\pi/2$ is an infeasible point for the actual problem. The additional qualification \eqref{qualificationInfeas} fails at $\pi/2$ because we cannot conclude that $y = 0$ from $0 = y \cos \pi/2$.

\section{Local Model and Implementation: Composite Proximal Bundle Method}\label{sec:bundle}

Algorithm~\ref{algo} of Section~\ref{sec:global} is conceptual as it omits instructions for how to obtain near-stationary points $x^\nu$ of the approximating problems \eqref{pbm-a} satisfying $\dist(0,S^\nu(x^\nu,y^\nu,z^\nu))\leq \epsilon^\nu$ for some $y^\nu,z^\nu\in \Re^m$. In this section, we give a local algorithm for this purpose. Combined with Algorithm~\ref{algo}, it results in a double-loop algorithm for the actual problem \eqref{pbm}.

To lighten the notation and stress the possibility that there might be {\em no} global approximation in a particular application, we omit the superscript $\nu$ and develop the local algorithm for the actual problem \eqref{pbm} under additional assumptions. However, a primary purpose of the algorithm is to obtain near-stationary points of the approximating problems \eqref{pbm-a}. 
%
We summarize the assumptions next. They are mainly motivated by the needs of an implementable algorithm. 

\begin{assumption}{\rm (composite proximal bundle method).}\label{assumptions}
\mbox{ }
\begin{enumerate}[label=(\alph*), start=1]
\item $X\subset \Re^n$ is nonempty, closed, and convex.
\item $f_0\fc{\Re^n}{\Re}$ is smooth and convex.
\item Each component function $f_1, \dots, f_m$ of $F$ is $L_i$-smooth on $X$, i.e., for each $i=1, \dots, m$, $f_i$ is smooth, $D_i(x)=\col{\nabla f_i(x)}$, and there exists $L_i>0$ such that
\[
\big\|\nabla f_i(x)-\nabla f_i(x')\big\|_2 \leq \tfrac{1}{2}L_i \|x-x'\|_2\quad \forall\, x,x'\in X.
\]

\item $h\fc{\Re^m}{\Re}$ is convex and there exists $L_h>0$ such that
\[
\big|h(z)-h(z')\big|\leq {L_h}\|z-z'\|_2\quad \forall\, z,z'\in \Re^m.
\]

\end{enumerate}
\end{assumption}

We stress that assumptions of this form might only need to hold for the approximating problems \eqref{pbm-a}. Thus, starting from the actual problem, we can deliberately construct global approximations that satisfy Assumption \ref{assumptions}. In Example~\ref{ex:buffered}, the global approximations $F^\nu$ satisfy Assumption~\ref{assumptions}(c) when $\psi_{ik}$, $i=1, \dots, s$, $k=1, \dots, r$, are twice smooth and an approximation $X^\nu$ of $X$ is compact because then $F^\nu$ is twice smooth by \cite[Example 4.16]{Royset_Wets_2021}. The global approximations $h^\nu$ satisfy Assumption~\ref{assumptions}(d) by virtue of being defined by sum and maxima of linear functions. We discuss Example~\ref{ex:ncvxSP} at the end of this subsection and Example~\ref{example:dist2set} in Section~\ref{sec:dc}.

The following algorithm, called Algorithm \ref{bundlemethod}, relies on local models of $h$ in \eqref{pbm} under Assumption~\ref{assumptions}, with certain properties of which convexity and lower boundedness are central. By following the lead of \cite{Oliveira_Sagastizabal_2014}, we say that 
\[
\text{$\modelh$ is a {\em lower model} of $h$ if it is a convex, real-valued function on $\reals^m$, and $\modelh(z) \leq h(z)~\forall z\in \reals^m$.}
\]

\begin{algorithm}[htb]
\caption{Composite Proximal Bundle Method}
\label{bundlemethod}
\begin{algorithmic}[1]
{\footnotesize
\State  Data: $x_0\in X$, $\kappa \in (0,1)$, $1<\tau $, $0<\underbar{t}<t_0<t_{\max}<\infty$, $\tol \geq 0$
\State Set $\hat x_0\gets x_0$ and construct a lower model $\modelh_0$ of $h$
\For{$k = 0, 1, 2, \dots$}
\State Compute $(x_{k+1}, y_{k+1}, z_{k+1})$ satisfying \label{line:xkk}
\begin{subequations}\label{eqn:MP}
\begin{align}
z_{k+1} &= F(\hat x_k)+\nabla F(\hat x_k)(x_{k+1}-\hat x_k)\label{eqn:MP1}\\
y_{k+1} &\in \partial \modelh_k(z_{k+1})\label{eqn:MP2}\\
-\nabla f_0(x_{k+1}) - \nabla F(\hat x_k)^\top y_{k+1} - \frac{x_{k+1}-\hat x_k}{t_k} &\in N_X(x_{k+1})\label{eqn:MP3} 
\end{align}
\end{subequations}
\State Set $v_k\gets f_0(\hat x_k)+h(F(\hat x_k)) - \big(f_0(x_{k+1}) + \modelh_k(z_{k+1})\big)$ \label{line:vk}
\If{$v_k\leq \tol$}
\State Stop and return $x_{k+1}$ 
\EndIf
\If{$f_0(x_{k+1})+ h(z_{k+1})\leq f_0(\hat x_k)+ h(F(\hat x_k))-  \kappa v_k$}\label{call-oracle-c}
\If{$f_0(x_{k+1})+h(F(x_{k+1}))\leq f_0(\hat x_k)+ h(F(\hat x_k))-  \frac{\kappa}{2} v_k$}\label{call-oracle-hc}
\State Set $\hat x_{k+1} \gets x_{k+1}$ and choose $t_{k+1}\in [t_k, \,t_{\max}]$\Comment{Serious Step}
\State Compute $\hat s_{k+1} \in \partial h(F(\hat x_{k+1}))$
\State Construct a lower model $\modelh_{k+1}$ of $h$ satisfying
\[
h(F(\hat x_{k+1}))+\inner{\hat s_{k+1}}{z-F(\hat x_{k+1})}
 \leq \modelh_{k+1}(z)~~~~~\forall z\in \reals^m
\]
\Else
\State Set $ \hat x_{k+1}\gets \hat x_k$, $t_{k+1}\gets t_k/\tau$, and $\underbar{t}\gets \min\col{t_{k}/\tau,\,\underbar{t}}$ \label{line:decrease_t} \Comment{Backtracking Step}
\State Compute $s_{k+1} \in \partial h(z_{k+1})$
\State Construct a lower model $\modelh_{k+1}$ of $h$ satisfying
\begin{equation}\label{model_req}
\max\begin{Bmatrix}
h(F(\hat x_k))+\inner{\hat s_{k}}{z-F(\hat x_k)}\\
h(z_{k+1})+\inner{s_{k+1}}{z-z_{k+1}}\\
\modelh_k(z_{k+1}) + \inner{y_{k+1}}{z - z_{k+1} }
\end{Bmatrix}
 \leq \modelh_{k+1}(z)~~~~~~\forall z\in \reals^m
\end{equation} 
\EndIf
\Else
\State Set $ \hat x_{k+1}\gets \hat x_k$ and choose $t_{k+1}\in [\underbar{t},\, t_k]$ \label{NS}  \Comment{Null Step}
\State Compute $s_{k+1} \in \partial h(z_{k+1})$
\State Construct a lower model $\modelh_{k+1}$ of $h$ satisfying \eqref{model_req}
\EndIf
\EndFor
}
  \end{algorithmic}
\end{algorithm}

There is much flexibility in how to construct lower models in bundle methods. In the convex setting, \cite{Correa_Lemarechal_1993} shows that any lower model that majorizes two specific linearizations suffices to ensure the desirable convergence properties. That analysis carries over to nonconvex settings as well \cite{Hare_Sagastizabal_2010,Oliveira_2019}. Algorithm \ref{bundlemethod} incorporates this flexibility and allows for a wide array of local models.

The main computational effort of Algorithm \ref{bundlemethod} takes place on Line~\ref{line:xkk}, which computes a global solution $(x_{k+1}, y_{k+1}, z_{k+1})$ of the convex {\em master problem} 
\begin{equation}\label{spbm}
\nnmin_{x\in X} f_0(x) + \modelh_k\big(F(\hat x_k)+\nabla F(\hat x_k)(x-\hat x_k)\big) + \frac{1}{2t_k}\norm{x-\hat x_k}^2
\end{equation} 
as expressed by the optimality condition \eqref{eqn:MP}. It is reasonable to assume that \eqref{spbm} is solved exactly because $\modelh_k$ is constructed to be a ``simple'' model of $h$, and both $h$ and $X$ might be ``simple,'' global approximations of some underlying quantities. Typically, $\epi \modelh_k$ and $X$ in \eqref{spbm} are polyhedral. 

As a concrete example, the local models $\modelh_k$ may be {\em cutting-plane models} of the form
\begin{equation}\label{cp}
\check{h}_k(z)=\max_{j\in B_k}\big\{h(z_j)+\inner{s_j}{z-z_j}\big\},
\end{equation}
with $B_k$ an index set. Then, \eqref{spbm} is equivalent to solving 
\begin{align}
\nnmin_{x\in X,\, (z,r) \in \Re^{m+1}} f_0(x) + r + \frac{1}{2t_k}\norm{x-\hat x_k}^2&\label{eqn:cutmodel}\\
\mbox{subject to } F(\hat x_k)+\nabla F(\hat x_k)(x-\hat x_k) & = z\nonumber\\
h(z_j) + \inner{s_j}{z-z_j} & \leq r~~\; \forall\, j\in B_k.\nonumber
\end{align}
If $X$ is polyhedral and $f_0$ is quadratic, then this problem is a convex quadratic program (QP), for which there are several efficient solvers. Regardless, \eqref{spbm} is a strictly convex optimization problem that can be addressed using a variety of methods found in the literature (e.g., convex bundle methods \cite[Chapters 11 \& 12]{WWBook}). Thus, $x_{k+1}$ and $z_{k+1}$ in Line~\ref{line:xkk} are readily available. 

As is standard in proximal bundle methods, the model subgradient $y_{k+1}$ can be computed straightforwardly as well. Let $\alpha_j \geq 0$ denote the Lagrange multiplier associated with the linear constraint $h(z_j) + \inner{s_j}{z - z_j} \leq r$ in \eqref{eqn:cutmodel}. By analyzing the dual problem, it can be shown that $\sum_{j \in B_k} \alpha_j = 1$ and that
$
y_{k+1} = \sum_{j \in B_k} \alpha_j s_j
$
is a subgradient of the cutting-plane model at $z_{k+1}$; see, for instance, \cite[\S\,11.5]{WWBook}.  

There is flexibility in the choice of \emph{information bundle} $B_k$ in \eqref{eqn:cutmodel} so that the requirement in~\eqref{model_req} holds. For instance, $B_k$ can contain only the three indexes related to the linearizations in \eqref{model_req}. This is an economical bundle that reduces the size of \eqref{eqn:cutmodel}. A full bundle amounts to keeping all linearizations in memory ($B_k=\col{0,1,\ldots,k}$) and also satisfies \eqref{model_req}. Numerical practice suggests that choosing $B_k$ between these two extreme cases is beneficial. A usual choice satisfying~\eqref{model_req} is $B_k=\bar B_k \cup\col{k+1}$, with $\bar B_k=\col{i \in B_k: \; \alpha_j >0 }$ being the index set of active constraints (linearizations) in \eqref{eqn:cutmodel}.

In addition to these cutting-plane models for defining $\modelh_k$, our setting is general enough to accommodate richer models such as quadratic \cite{Astorino_2011} or spectral \cite{Helmberg_2000} ones. Subsection~\ref{sec:special} presents an alternative local model for the case when $h$ is a composite function, and explains how it can be updated to meet the requirements in \eqref{model_req}.

Most lines after Line~\ref{line:xkk} contribute to the construction of the local models. They ensure that $\modelh_k$ satisfies, for all $z \in \Re^m$,
\[
 h\big(F(\hat x_k) + \inner{\hat s_k}{z-F(\hat x_k)}\big) \leq \modelh_k(z) \leq h(z)\quad \forall\, k. 
\]
In particular, the model $\modelh_k$ coincides with $h$ at $z=F(\hat x_k)$:
\begin{equation}\label{model=c}
f_0(\hat x_k)+\modelh_k\big(F(\hat x_k)\big) = f_0(\hat x_k)+ h\big(F(\hat x_k)\big).
\end{equation}
Line~\ref{line:vk} defines the \emph{predicted decrease} $v_k$, which estimates how much the new iterate $x_{k+1}$ is expected to decrease the objective function. Specifically, $\hat x_k$ is feasible in \eqref{spbm}, which implies
\[
f_0(x_{k+1})+\modelh_k\big(F(\hat x_k)+\nabla F(\hat x_k)(x_{k+1}-\hat x_k)\big) + \frac{1}{2t_k}\norm{x_{k+1}-\hat x_k}^2\leq f_0(\hat x_{k})+h\big(F(\hat x_k)\big).
\]
Together with~\eqref{model=c}, this inequality ensures the following key relation:
\begin{equation}\label{vk}
\frac{1}{2t_k}\norm{x_{k+1}-\hat x_k}^2 \leq  v_k \quad \forall\, k.
\end{equation}
Since $F$ is linearized at $\hat{x}_k$ and $h$ is replaced by the lower model $\modelh_k$, the predicted decrease $v_k$ may not reliably estimate the quantity
\[
f_0(\hat{x}_k) + h\big(F(\hat{x}_k)\big) - \Big(f_0(x_{k+1}) + h\big(F(x_{k+1})\big)\Big).
\]
For this reason, Algorithm~\ref{bundlemethod} classifies iterates as \emph{serious}, \emph{null}, or \emph{backtracking}. An iterate is declared \emph{serious} when the actual decrease is at least a fraction $\kappa/2$ of the predicted decrease. Otherwise, the model must be refined, and the prox-parameter $t_k$ is not allowed to increase (it is indeed decreased at backtracks), with the aim of making the predicted decrease $v_k$ closer to the actual decrease. 
%
 More specifically, if the first test at Line~\ref{call-oracle-c} is successful, then
 two possibilities arise: Algorithm~\ref{bundlemethod} either declares a serious step (in the sense that \( x_{k+1} \) decreases the objective function), or a backtracking step. The interpretation of the latter is as follows: while the model \( \modelh_k \) provides a sufficiently accurate approximation of \( h \) around $F(\hat{x}_k)$, the composed model \( \modelh_k(F(\hat{x}_k) + \nabla F(\hat{x}_k)(\,\cdot\, - \hat{x}_k)) \) still fails to adequately approximate $h\circ F$ around $\hat x_k$. The algorithm thus decreases the prox-parameter to generate new trial points in the vicinity of the best solution candidate \( \hat{x}_k \), where the approximation error is smaller (recall that \( F \) consists of \( L_i \)-smooth functions under Assumption~\ref{assumptions}). 
If the test at Line~\ref{call-oracle-hc} is satisfied, then the iterate $x_{k+1}$ becomes the new stability center and the model $\modelh_{k+1}$ can be as simple as the linearization $h(F(\hat x_{k+1}))+\inner{\hat s_{k+1}}{\cdot-F(\hat x_{k+1})}$. On the other hand, if this test fails, condition~\eqref{model_req} requires that the updated model not only overestimates the linearization of $ h$ at the last stability center (this linearization is available from earlier computations), but also the linearizations of both $h$ and $\modelh_k$ at the point $z_{k+1}$. 

The predicted decrease $v_k$ also serves as an optimality measure. By \eqref{vk}, $v_k=0$ implies that $x_{k+1}=\hat x_k$. Then, $\hat x_k$ solves the master problem \eqref{spbm} and its optimality condition \eqref{eqn:MP} reduces to
\[
-\nabla f_0(\hat x_{k}) - \nabla F(\hat x_k)^\top y_{k+1} \in N_X(\hat x_{k}),\quad\mbox{with }\; y_{k+1} \in \partial \modelh_k\big(F(\hat x_k)\big).
\]
This actually implies that $\hat x_k$ is stationary for \eqref{pbm} under Assumption~\ref{assumptions} because $\modelh_k$ is a lower model that coincides with $h$ at $F(\hat x_k)$ and thus $\partial \modelh_k(F(\hat x_k)) \subset \partial h(F(\hat x_k))$. Consequently, $0\in S(\hat x_k,y_{k+1},F(\hat x_k))$ with $S$ given in~\eqref{eqn:S} (recall that $D_i(x)=\col{\nabla f_i(x)}$ in Assumption~\ref{assumptions}(c)). The condition $v_k\leq \tol$ for some $\tol\geq 0$ is therefore a valid stopping test provided that $\col{t_k, k\in\nats}$ is bounded away from zero.  In Algorithm~\ref{bundlemethod}, the prox-parameter $t_k$ satisfies $t_k\geq \underbar{t}$ for all $k$, with
$\underbar{t}$ being potentially decreased during backtracking steps. We show in Lemma~\ref{lem:bound_tk} that $\underbar{t}$ and, thus, $t_k$ are bounded away from zero.

While we omit the details, Algorithm~\ref{bundlemethod} is easily extended to handle a variable prox-metrics that replaces $\frac{1}{2t_k} \norm{x - \hat{x}_k}^2$ by a quadratic form involving a positive definite matrix; see \cite{Sagastizabal_2013} for such enhancements.

\begin{example}{\rm (distributionally robust two-stage nonconvex programs (cont.)).}\label{ex:ncvxSP2} The function $h$ in Example~\ref{ex:ncvxSP} is real-valued, convex, and satisfies the Lipschitz condition in Assumption~\ref{assumptions}(d) because $\partial h(u) = \nargmax_{p\in \mathcal{P}} \langle p,u\rangle\subset \mathcal{P}$ and $\mathcal{P}$ is compact. 

Under additional assumptions, $f_i$ is lLc and its global approximation $f_i^\nu$ in  Example~\ref{ex:ncvxSP} satisfies Assumption~\ref{assumptions}(c) and also \eqref{eqn:suffFconv}. To simplify the notation, we omit the subscript $i$ and consider the function $f:\reals^n\to \reals$ given by 
\[
f(x) = \inf_{y \in Y}\big\{ \phi(x,y)~\big|~ \psi(x,y)\leq 0 \big\}, 
\]
where $Y \subset \Re^{m}$ is nonempty, convex, and compact and both $\phi:\Re^n\times \Re^m \to \Re$ and $\psi:\Re^n\times \Re^m \to \Re^{q}$ are smooth. Moreover, suppose that at each $\bar x\in \reals^n$ the following hold: (i) $\phi(\bar x,\cdot)$ and each component function of $\psi(\bar x,\cdot)$ are convex; $\{y\in Y~|~\psi(\bar x,y)\leq 0\}\neq\emptyset$; and for all $\bar y \in \nargmin_{y\in Y}\{\phi(\bar x,y)\,|\, \psi(\bar x,y)\leq 0 \}$, the qualification holds: 
\begin{align}\label{qual:infproj}
v\in N_{(-\infty,0]^q}\big(\psi(\bar x, \bar y)\big) ~\mbox{ and }~ -\nabla_y \psi(\bar x, \bar y)^\top v \in N_{Y_i}(\bar y)~~~\Longrightarrow~~~v = 0.
\end{align}
Under these conditions, $f$ is lLc. This fact follows from \cite[Corollary 10.14(a)]{Rockafellar_Wets_1998} when applied to $f$ at an arbitrary point $\bar x\in\reals^n$. The corollary indeed applies because for $\bar y \in \nargmin_{y\in Y}\{\phi(\bar x,y)\,|\, \psi(\bar x,y)\leq 0 \}$ one has
\[
(z,0) \in \partial^\infty g(\bar x, \bar y)  ~~\Longrightarrow~~ z = 0, 
\]
where $g$ is the function given by $g(x,y) = \phi(x,y) + \iota_Y(y) + \iota_{(-\infty,0]^q}(\psi(x,y))$. The necessary calculus rules (see, e.g., Example 10.8 in \cite{Rockafellar_Wets_1998}) apply under qualification \eqref{qual:infproj}. We observe that this qualification holds, for example, if $Y= \reals^m$ and the constraints $\psi(\bar x, y)\leq 0$ satisfy the Slater condition; see, e.g., \cite[Example 5.47]{Royset_Wets_2021}.  

Next, we discuss the approximating function $f^\nu:\reals^n\to \reals$ and its relation to $f$ under the assumption that there is no $x$-dependent constraints and $\phi(x,y) = \tfrac{1}{2}\langle y, Q(x)y\rangle + \langle c(x), y\rangle$, i.e.,  
\[
f(x) = \inf_{y\in Y} \tfrac{1}{2}\big\langle y, Q(x)y\big\rangle + \big\langle c(x), y\big\rangle, ~~~~~f^\nu(x) = \inf_{y \in Y} \tfrac{1}{2}\big\langle y, Q(x)y\big\rangle + \big\langle c(x), y\big\rangle + \tfrac{1}{2}\eta^\nu\norm{y}^2,
\]
where $\eta^\nu\in (0,\infty)$, $Q:\reals^n\to \reals^{m\times m}$ and $c:\reals^n\to \reals^m$ are smooth, $Q(x)$ is positive semidefinite for all $x\in\reals^n$, and there is $M\in [0,\infty)$ such that 
\[
\big\|Q(x) - Q(\bar x)\big\|_F \leq M\|x- \bar x\|_2~~~ \text{ and }~~~ \big\|c(x) - c(\bar x)\big\|_2 \leq M\|x- \bar x\|_2~~~\forall x,\bar x\in X.
\] 
Here $\|\cdot\|_F$ is the Frobenius norm. We can then invoke \cite[Proposition 6.30]{Royset_Wets_2021} to conclude that 
\[
\bar \partial f(x) = \con\big\{ \nabla_x \phi(x,\bar y) ~\big|~ \bar y \in \nargmin_{y\in Y} \phi(x,y) \big\}, ~~~
\partial f^\nu(x)  = \Big\{ \nabla_x \phi\big(x,y^\nu(x)\big) \Big\}
\]
for the unique maximizer $y^\nu(x)$ solving $-(Q(x) + \eta^\nu I)y - c(x) \in N_Y(y)$, where $I$ is the $m\times m$ identity matrix. It now follows straightforwardly that \eqref{eqn:suffFconv} holds when $\eta^\nu \to 0$. Moreover, if $\nabla Q$ and $\nabla c$ are lLc and $X$ is bounded, then $f^\nu$ is smooth and also satisfies Assumption~\ref{assumptions}(c).\mybox 
\end{example}

\subsection{Convergence Analysis}
Let $N^a \subset \nats$ be the index set of iterations generated by Algorithm \ref{bundlemethod}, which may be finite if the algorithm stops after finitely many steps or simply be $\nats$ otherwise. Under Assumption~\ref{assumptions} and minor boundedness conditions, we establish next that: if $N^a$ is infinite, all cluster points of the  sequence $\col{\hat x_k, k\in N^a}$ produced by Algorithm~\ref{bundlemethod} are stationary; if $N^a$ is finite, the last point $\hat x_k$ is stationary. Moreover, if the tolerance $\tol$ in Algorithm~\ref{bundlemethod} is positive and the algorithm is applied to an approximating problem \eqref{pbm-a}, then the algorithm stops after a finite number of iterations with a triplet $(x,y,z)$ satisfying $\dist(0, S^\nu(x,y,z)) \leq \epsilon^\nu$ for some $\epsilon^\nu \in [0,\infty)$ that vanishes if $\tol$ tends to zero. Thus, Algorithm~\ref{bundlemethod} furnishes an approach for implementing Line 3 in Algorithm \ref{algo}. 

The first boundedness condition pertain merely to a bounded level-set: 
We say that \eqref{pbm} is {\em level-bounded for} $x_0\in X$ when 
\[
\Big\{x \in X\,\Big|\, f_0(x) +h\big(F(x)\big)\leq  f_0(x_0) + h\big(F(x_0)\big)\Big\} ~~~\text{ is bounded}.
\]
The second boundedness condition relates to the size of the subgradients of $\modelh_k$. We say that the models $\{\modelh_k, k\in N^a\}$ have {\em uniformly bounded subgradients on bounded sets} when 
\begin{equation}\label{uniform_bound}
    \mbox{for any bounded $C \subset \Re^m$, $\exists M_C>0$ such that } \norm{y}\leq M_C ~~\forall y\in \partial \modelh_k(z), z\in C, k \in N^a.
\end{equation}
This is a natural assumption in our setting because \( h \) is real-valued and convex under Assumption~\ref{assumptions} and its subdifferential is therefore locally bounded. Any meaningful model $\modelh_k$ for $h$ should also exhibit this property. This is, for instance, the case of the cutting-plane model~\eqref{cp}, as its subdifferential is the convex hull of subgradients of $ h$ at previous points; Subsection~\ref{sec:special} provides another example.

%

\begin{theorem}{\rm (convergence of Algorithm~\ref{bundlemethod}).}\label{thm:conv} Under Assumption~\ref{assumptions}, suppose that Algorithm~\ref{bundlemethod} is applied to an instance of \eqref{pbm} that is level-bounded for $x_0$. 

If $\tol=0$, then either the algorithm stops on Line 7 in iteration $k$ with $x^\star = \hat x_k = x_{k+1}$ satisfying 
\begin{equation}\label{eqn:optcond1}
0\in S(x^\star,y^\star,z^\star) \text{ for some } y^\star\in \reals^m \mbox{ and } z^\star=F(x^\star)
\end{equation}
or the algorithm continues indefinite. In the latter case, if the resulting models $\{\modelh_k, k\in \nats\}$ have uniformly bounded subgradients on bounded sets, then the iterates $\col{x_k,k\in \nats}$, stability centers $\col{\hat x_k,k\in \nats}$, and predicted decreases $\{v_k,k\in \nats\}$ are well-defined, $\nliminf v_k = 0$, and every cluster point $x^\star$ of $\col{\hat x_k,k\in \nats}$ satisfies \eqref{eqn:optcond1}.

If $\tol>0$, then the algorithm stops on Line 7 after a finite number of iteration $k_\tol$ and the last iterate $ x_\tol=x_{k_\tol+1}$ is approximately stationary by satisfying
\[
\dist\big(0, S(x_\tol,y^\star ,z^\star)\big) \leq \epsilon_\tol,
\]
and $\epsilon_\tol\geq 0$ is such that $\epsilon_\tol \to 0$ as $\tol \to 0$.
\end{theorem}

When $\tol > 0$, the above theorem asserts that the last iterate $x_{k+1}$ computed by the algorithm satisfies~\eqref{eqn:tol} for some $y^\star$, $z^\star$, and $\epsilon_\tol$. None of these three elements need to be computed; their mere existence suffices for the analysis of Algorithm~\ref{algo}. Additional comments can be found in Remark~\ref{rem:WeDontComputeEps} below.
To prove Theorem~\ref{thm:conv} we need a series  of auxiliary results. Throughout this subsection, we employ the notation, with $\col{x_k, k \in N^a}$ the iterates produced by Algorithm~\ref{bundlemethod}:
\begin{subequations}\label{zjj}
\begin{align}
z_{j+1}& =F(\hat x_j)+\nabla F(\hat x_j) (x_{j+1}-\hat x_j), \quad s_{j+1} \in \partial h(z_{j+1}), \\
\hat z_j& = F(\hat x_j), \mbox{ and } \hat s_j \in \partial h(\hat z_j), \quad j = 0, 1, \ldots, k.
\end{align}
\end{subequations}
We start by relating subgradients of $\modelh_k$ with those of $h$.

\begin{lemma}\label{lem:approx-sub} Suppose that Assumption~\ref{assumptions} holds. If $y_{k+1}\in \partial \modelh_k(z_{k+1})$, then $y_{k+1}\in \partial_{e_k} h(\hat z_{k})$, with 
\[
e_k = h(\hat z_{k}) - \big(\modelh_k(z_{k+1}) + \inner{y_{k+1}}{\hat z_k - z_{k+1}}\big) \geq 0.
\]
Let $N\subset N^a$ be an index set. If $\col{z_k,k \in N}$ is bounded and \eqref{uniform_bound} holds, then there exists $M>0$ such that $e_k\leq M\norm{z_{k+1}-\hat z_k}$ for all $k\in N$.
\end{lemma}
\begin{proof}
The subgradient inequality gives, for all $z \in \Re^m$,
\[
h(z)\geq \modelh_k(z)\geq \modelh_k(z_{k+1}) + \inner{y_{k+1}}{z - z_{k+1}}
= h(\hat z_{k}) + \inner{y_{k+1}}{z - \hat z_{k}} - e_k.
\]
Nonnegativity of $e_k$ follows from the inequality $\modelh_k\leq h$ and the fact that $y_{k+1}\in \partial \modelh_k(z_{k+1})$. Furthermore, if $\col{z_k,k\in N}$ is contained in a bounded set $C$, condition~\eqref{uniform_bound} ensures the existence of $M'>0$ such that all subgradients of $\modelh_k$ at any $z\in C$ are bounded by $M'$. As a result, $\norm{y_{k+1}}\leq M'$ for all $k\in N$ and 
$|\modelh_k(z)-\modelh_k(z')|\leq M'\norm{z-z'}$ for all $z,z'\in C$ and all $k\in N$.
As $h(\hat z_{k})=\modelh_k(\hat z_{k})$  (cf. \eqref{model=c}), we have that
\begin{align*}
e_k&= h(\hat z_{k}) - [\modelh_k(z_{k+1}) + \inner{y_{k+1}}{\hat z_k - z_{k+1}}]\\
&= h(\hat z_{k})-\modelh_k(\hat z_{k}) - [\modelh_k(z_{k+1})-\modelh_k(\hat z_{k}) + \inner{y_{k+1}}{\hat z_k - z_{k+1}}] \\
&\leq |\modelh_k(z_{k+1})-\modelh_k(\hat z_{k})|+\norm{y_{k+1}}\norm{z_{k+1} -\hat z_k}\\
& \leq 2M'\norm{z_{k+1} -\hat z_k}\,. 
\end{align*}
The result follows with $M=2M'$.
\mybox
\end{proof}

The following lemma is used to prove Lemma~\ref{lem:bound_tk} below. 

\begin{lemma}\label{lem:small_tk}
Under Assumption~\ref{assumptions}, let $L_F:= \sqrt{\sum_{i=1}^m L_i^2}$.
If $t_k \in (0,\frac{\kappa}{2L_hL_F})$ and 
\begin{equation}\label{armijo2}
f_0(x_{k+1})+ h(z_{k+1})\leq f_0(\hat x_{k})+ h(F(\hat x_k))-  \kappa v_k,
\end{equation}  then the descent test
$
f_0(x_{k+1}) +h(F(x_{k+1})) \leq f_0(\hat x_{k})+ h(F(\hat x_k))-  \frac{\kappa}{2} v_k
$ 
holds.
\end{lemma}
\begin{proof}
By Assumption~\ref{assumptions}(c) and the mean value theorem, we obtain that 
\begin{equation}\label{L-smooth}
\big\|F(x')+\nabla F(x') (x-x') - F(x)\big\|_2\\
\leq \displaystyle \tfrac{1}{2}L_F\norm{x- x'}^2\quad \forall x,x' \in X.
\end{equation}
Items (c) and (d) in Assumption~\ref{assumptions} give, thanks to \eqref{L-smooth}, the relation
\begin{align*}
|h(F(\hat x_k)+\nabla F(\hat x_k) (x_{k+1}-\hat x_k))-h(F(x_{k+1}))|&\leq L_h\norm{F(\hat x_k)+\nabla F(\hat x_k) (x_{k+1}-\hat x_k) - F(x_{k+1})}\\
&\leq \frac{1}{2}L_h L_F\norm{x_{k+1}-\hat x_k}^2,
\end{align*}
i.e.,
$
h(F(x_{k+1})) \leq h(F(\hat x_k)+\nabla F(\hat x_k) (x_{k+1}-\hat x_k)) +\frac{L_h L_F}{2}\norm{x_{k+1}-\hat x_k}^2 .
$
Adding $f_0(x_{k+1})$ to both sides of this inequality and recalling \eqref{armijo2} we obtain
\begin{align*}
f_0(x_{k+1}) +h(F(x_{k+1}))& \leq f_0(x_{k+1})+ h(F(\hat x_k)+\nabla F(\hat x_k) (x_{k+1}-\hat x_k)) +\frac{L_h L_F}{2}\norm{x_{k+1}-\hat x_k}^2 \\
&\leq f_0(\hat x_{k})+ h(F(\hat x_k)) -\kappa  v_k + \frac{L_h L_F}{2}\norm{x_{k+1}-\hat x}^2.
\end{align*}
We now use~\eqref{vk} to get
\begin{align*}
f_0(x_{k+1}) +h(F(x_{k+1}))& \leq 
f_0(\hat x_{k})+ h(F(\hat x_k)) + (L_h L_Ft_k-\kappa)v_k.
\end{align*}
As $0<t_k\leq \frac{\kappa}{2L_hL_F}$, then $(L_h L_Ft_k-\kappa) \leq -\frac{\kappa}{2}$ and the
result follows. \mybox 
\end{proof}
\begin{lemma}\label{lem:bound_tk}
Consider Algorithm~\ref{bundlemethod} applied to problem~\eqref{pbm} under Assumption~\ref{assumptions}.
Then the sequence of prox-parameters yielded by the algorithm is bounded away from zero: $t_k \in [t_{\min},\,t_{\max}]$ for all $k$, with $t_{\min}= \frac{\kappa}{2\tau L_hL_F}>0$ and $t_{\max}$ is given.
\end{lemma}
\begin{proof}
Observe that $t_k\geq \underbar{t}$ for all $k$, and $\underbar{t}$ is only reduced at Line~\ref{line:decrease_t} whenever $t_k/\tau < \underbar{t}$. 
We claim the algorithm cannot decrease $\underbar{t}$ even further once $\underbar{t}\leq \frac{\kappa}{2L_hL_F}$. Indeed, in this case, if the algorithm decreases $t_k$ so that $t_k \in [\underbar{t},\, \frac{\kappa}{2L_hL_F}]$, then by Lemma~\ref{lem:small_tk} a new serious step is produced. This prevents from decreasing $\underbar{t}$ even further.
Therefore, due to the rule used at Line~\ref{line:decrease_t} of the algorithm, we conclude that $\underbar{t}\geq \frac{\kappa}{2\tau L_hL_F}$, and the result follows.
\mybox
\end{proof}
Given that the sequence of prox-parameters is bounded away from zero, we now show that $v_k$ can be seen as a measure of stationarity, and thus $v_k\leq \tol$ is a valid stopping test.

\begin{proposition}{\rm (vanishing tolerance in Algorithm~\ref{bundlemethod}).}\label{prop:vk}
Consider Algorithm~\ref{bundlemethod} with $\tol=0$ applied to problem~\eqref{pbm} under Assumption~\ref{assumptions}.
Suppose that~\eqref{uniform_bound} holds and there exists an infinite index set $N \subset N^a$ such that
\[
\hat x_{k} \Nto  x^\star \quad \text{ and } \quad  v_k \Nto 0.
\]
Then $x_{k+1} \Nto x^\star$ and the subsequence $\col{y_{k+1}, k\in N}$ is bounded. 
Let $y^\star$ be a cluster point of this subsequence, $z^\star= F( x^\star)$,  and define
  \begin{align*}
\epsilon_{k+1}
&=\sqrt{\norm{F(x_{k+1})-z^\star}^2 +\frac{1}{t_{\min}^2}\norm{x_{k+1}-\hat x_k}^2  + \norm{\nabla F(x_{k+1})^\top y^\star - \nabla F(\hat x_{k})^\top  y_{k+1} }^2}.
\end{align*}
Then $\dist(0,S(x_{k+1},y^\star,z^\star)) \leq {\epsilon_{k+1}}$,
$\nliminf_{k\in N} \epsilon_{k+1}=0$, and $x^\star $ is stationary, i.e., $0\in S(x^\star, y^\star,  z^\star)$. 
\end{proposition}
\begin{proof}
It follows from~\eqref{vk} and the fact that $t_k\leq t_{\max}<\infty$ for all $k$ that the condition $v_k \Nto 0$
implies $\lim_{k \in N} \hat x_{k} =  \lim_{k \in N} x_{k+1}= x^\star \in X$.
From~\eqref{eqn:MP1}, continuity of $F$ and $\nabla F$, we get that
\begin{align*}
\lim_{k \in N} z_{k+1}=\lim_{k \in N} [F(\hat x_k)+\nabla F(\hat x_k) (x_{k+1}-\hat x_k)] = F(x^\star) = z^\star\\
\lim_{k \in N} \hat z_{k}=\lim_{k \in N} F(\hat x_k)=\lim_{k \in N} F( x_{k+1})  = z^\star\\
\lim_{k \in N} \nabla F(\hat x_k)=\lim_{k \in N} \nabla F(x_{k+1}) = \nabla F(x^\star).
\end{align*}
As a result, both subsequences $\col{z_{k+1}, k \in N}$ and $\col{\hat z_{k}, k \in N}$ are bounded, and Lemma~\ref{lem:approx-sub}, under condition~\eqref{uniform_bound}, assures the existence of a constant $M>0$ such that $y_{k+1}\in \partial_{a_k} h(\hat z_{k})$ for all $k \in N$, with $a_k=M\norm{z_{k+1}-\hat z_k}$. Thus $\col{y_{k+1}, k\in N}$ is bounded and has at least a cluster point $y^\star$. The outer-semicontinuity of the approximate subdifferential asserts that  $y^\star \in  \partial h(z^\star)$. We have thus shown that  $\nliminf_{k\in N} \epsilon_{k+1}=0$.
As $D_i(x)=\col{\nabla f_i(x)}$ by Assumption~\ref{assumptions}, observe that
\begin{align*}
\dist^2(0, S(x_{k+1},y^\star,z^\star)) &
\leq \norm{F(x_{k+1})- z^\star}^2 + \dist^2(0,\partial h(z^\star) - y^\star) +\\
&\quad \dist^2(0,\nabla f_0(x_{k+1})+ \nabla F(x_{k+1})^\top  y^\star  + N_X( x_{k+1})),
\end{align*}
and $\dist^2(0,\partial h( z^\star) -y^\star) =0$. By employing~\eqref{eqn:MP3}, we can derive the following bound for the last term above:
\begin{align*}
 & \dist^2(0,\nabla f_0(x_{k+1})+ \nabla F(x_{k+1})^\top  y^\star  + N_X( x_{k+1})) \\
 &\leq \dist^2(0, -\frac{x_{k+1}-\hat x_k}{t_k}+\nabla F(x_{k+1})^\top  y^\star - \nabla F(\hat x_{k})^\top  y_{k+1}  ) \\
 &\leq \frac{1}{t_{\min}^2}\norm{x_{k+1}-\hat x_k}^2 + \norm{\nabla F(x_{k+1})^\top  y^\star - \nabla F(\hat x_{k})^\top  y_{k+1} }^2.
\end{align*}
This shows that $\dist^2(0, S(x_{k+1}, y^\star, z^\star)) \leq \epsilon_{k+1}$.
As $\nliminf_{k\in N} \epsilon_{k+1}=0$, we conclude that $0 \in S(x^\star, y^\star, z^\star)$.
\mybox
\end{proof}
\begin{remark}\label{rem:WeDontComputeEps}
Under the assumptions of Proposition~\ref{prop:vk}, if Algorithm~\ref{bundlemethod} stops at iteration $k$ with $v_k \leq \tol$, the point $x^\nu := x_{k+1}$ satisfies \eqref{eqn:tol} of Algorithm~\ref{algo} with $y^\nu := y^\star$, $z^\nu := z^\star$, and $\epsilon^\nu := \epsilon_{k+1}$. Neither $y^\nu$, $z^\nu$, nor $\epsilon^\nu$ is actually computed by Algorithm~\ref{bundlemethod} at iteration $k$. This does not pose an issue, as Algorithm~\ref{algo} only requires the existence of $y^\nu$, $z^\nu$, and $\epsilon^\nu$, and that the error $\epsilon^\nu$ asymptotically vanishes. Proposition~\ref{prop:vk} shows the latter is achieved by driving $\tol$ to zero. 
\end{remark}
We are now ready to conclude the analysis of Algorithm~\ref{bundlemethod}.

\begin{proof} {\bf of Theorem~\ref{thm:conv}}.
 If the algorithm stops at iteration $k$ with $\tol = 0$, then $v_k = 0$, which gives $x_{k+1}=\hat x_k$.
With the notation in~\eqref{zjj}, $z_{k+1}  =F(\hat x_k)+\nabla F(\hat x_k) (x_{k+1}-\hat x_k)=   F(\hat x_k)=\hat z_k$. Denote $z^\star = z_{k+1}$ and $y^\star = y_{k+1}\in \partial h(\hat z_k)$ ($= \partial h(z^\star)$).
 As a result, $\epsilon_{k+1}=0$ in Proposition~\ref{prop:vk} and we get $0\in S(x_{k+1},y_{k+1},z_{k+1})=S(\hat x_k,y_{k+1},F(\hat x_k))$,
i.e, $x_{k+1}=\hat x_k$ is stationary to problem~\eqref{pbm} under Assumption~\ref{assumptions}.

In what follows we suppose that the algorithm loops forever, i.e., $N^a=\nats$. There are two possibilities: the algorithm produces infinitely or finitely many serious steps.

\noindent \underline{Infinitely many serious steps.}
Let $\hat N \subset \nats$ be the index set
containing the iterations at which the descent test is satisfied:
\[
f_0(x_{k+1}) +h(F(x_{k+1})) \leq f_0(\hat x_{k})+ h(F(\hat x_k))-  \frac{\kappa}{2} v_k\quad \forall\, k \in \hat N.
\]
In this case $x_{k+1}$ becomes a stability center and the above inequality holds with $x_{k+1}$ replaced with $\hat x_{k+1}$.
Thus, $\col{\hat x_k, k \in \nats}$ lies in the bounded level-set $\col{x \in X\,|\, f_0(x) +h(F(x))\leq  f_0(x_0) +h(F(x_0))}$ (as $\hat x_0=x_0$). The sequence $\col{\hat x_k, k \in \nats}$  is thus bounded and a telescopic sum provides
\[
0\leq  \frac{\kappa}{2} \sum_{k \in \hat N}v_k\leq  
f_0(x_0) +h(F(x_0))-\lim_{k \in \hat N} [f_0(\hat x_{k+1}) +h(F(\hat x_{k+1}))] <\infty,
\]
showing that $\lim_{k\in \hat N } v_k= 0$. Since $t_k\geq t_{\min}$ (see Lemma~\ref{lem:bound_tk}), we have from~\eqref{vk} that $\lim_{k\in \hat N} (x_{k+1}-\hat x_k)=0$. For any cluster point $x^\star $ of $\col{\hat x_k, k \in \nats}\subset X$, we can thus extract an infinite index set $N \subset \hat N$ such that $\lim_{k\in N} x_{k+1}=x^\star=\lim_{k\in N} \hat x_{k}$. 
We can now rely on Proposition~\ref{prop:vk} to conclude that 
$0\in S(x^\star, y^\star, z^\star)$, with $ y^\star$ a cluster point of $\col{y_{k+1}, k \in N}$ and $z^\star=F( x^\star)$:  $x^\star$ is stationary to problem~\eqref{pbm} under Assumption~\ref{assumptions}.


\noindent \underline{Finitely many serious steps.} 
In this case, we first claim that inequality~\eqref{armijo2} holds only finitely many times, i.e., there are only finitely many backtracking steps. Otherwise, after the last serious step, the rule to update the prox-parameter  would drive $t_k$ to zero, which is impossible by Lemma~\ref{lem:bound_tk}. 
Therefore,
inequality~\eqref{armijo2} is never satisfied for $k$ large enough and, after a certain iteration, only null steps are produced by the algorithm. 

We will show that the last stability center $\hat x$ is stationary and $\lim_{k\to \infty} v_k=0$. To this end,
let $\hat k$ be the iteration at which the last stability center was determined. Thus, $\hat x = \hat x_{\hat k}$  is fixed and $t_{k+1}\leq t_k$ for all $k > \hat k$. The algorithm's master program becomes
\[
\{x_{k+1}\}=\nargmin_{x \in X} f_0(x) + \modelh_k\big(F(\hat x)+\nabla F(\hat x) (x-\hat x)\big) + \frac{1}{2t_k}\norm{x-\hat x}^2.
\]
We claim that, in this case, the analysis of Algorithm~\ref{bundlemethod} reduces to that of finitely many serious steps issued by a standard convex bundle method, furnished with the descent test~\eqref{armijo2}, applied to the convex problem
\begin{equation}\label{aux1}
\nnmin_{x \in X}\; f_0(x)+ w(x), \quad \mbox{with}\quad w(x)=h(F(\hat x)+\nabla F(\hat x) (x-\hat x)).
\end{equation}
Indeed, by defining $\mathfrak{w}_k(x)= \modelh_k(F(\hat x)+\nabla F(\hat x) (x-\hat x))$ for all $k$, we have that for all $x$ and $z=F(\hat x)+\nabla F(\hat x) (x-\hat x)$, condition~\eqref{model_req} yields
\begin{align*}
\mathfrak{w}_{k+1}(x)&=\modelh_{k+1}(z)\\
&\geq h(z_{k+1})+\inner{s_{k+1}}{z-z_{k+1}}\\
&=h(F(\hat x)+\nabla F(\hat x) (x_{k+1}-\hat x))+\inner{s_{k+1}}{F(\hat x)+\nabla F(\hat x) (x-\hat x) - [F(\hat x)+\nabla F(\hat x) (x_{k+1}-\hat x)]}\\
&=w(x_{k+1})+\inner{s_{k+1}}{\nabla F(\hat x) (x-x_{k+1})}=w(x_{k+1})+\inner{\nabla F(\hat x)^\top s_{k+1}}{x-x_{k+1}}\\
&=w(x_{k+1})+\inner{s^w_{k+1}}{x-x_{k+1}},
\end{align*}
with $\nabla F(\hat x)^\top s_{k+1}=:s^w_{k+1} \in \partial w(x_{k+1})$. Furthermore, due to \eqref{model_req},
\begin{align*}
\mathfrak{w}_{k+1}(x)=\modelh_{k+1}(z)
&\geq \modelh_k(z_{k+1})+\inner{y_{k+1}}{z-z_{k+1}}\\
&= \mathfrak{w}_k(x_{k+1})+\inner{\nabla F(\hat x)^\top y_{k+1}}{x-x_{k+1}}\\
&=\mathfrak{w}_k(x_{k+1})+\inner{s^{\mathfrak{w}}_{k+1}}{x-x_{k+1}},
\end{align*}
with $\nabla F(\hat x)^\top y_{k+1}=:s^{\mathfrak{w}}_{k+1} \in \partial \mathfrak{w}_k(x_{k+1})$. Hence, thanks to condition~\eqref{model_req}, we have that $\mathfrak{w}_{k+1}(x)\geq \max\col{w(x_{k+1})+\inner{s^w_{k+1}}{x-x_{k+1}},\,\mathfrak{w}_k(x_{k+1})+\inner{s^{\mathfrak{w}}_{k+1}}{x-x_{k+1}}}$, satisfying thus the minimal conditions a model must satisfy to ensure convergence of a standard proximal bundle  method applied to problem~\eqref{aux1} (see, for instance, \cite[\S 4]{Correa_Lemarechal_1993}, \cite[\S 4.1]{vanAckooij_Bello-Cruz_Oliveira_2016}, and discussions in \cite[\S 3]{Frangioni_Gorgone_2014}). Hence, as \eqref{armijo2} (the descent test to problem~\eqref{aux1}) does not hold for $k$ large enough, Theorem 4.4 of \cite{Correa_Lemarechal_1993}) asserts that 
\[
\lim_{k\to \infty} v_k =0,\quad \lim_{k \to\infty}x_{k+1}=\hat x, \quad \mbox{and $\hat x$ solves~\eqref{aux1}.}
\]
The latter condition implies $0 \in S(\hat x,\hat y,\hat z)$, with $\hat z= F(\hat x)$ and $\hat y \in \partial h(\hat z)$.

In both cases (finitely or infinitely many serious steps), the algorithm ensures the existence of an index set $N\subset \nats$ such that  $ \lim_{k \in N}  x_{k+1} = \lim_{k \in N} \hat x_{k}  = x ^\star\in X$, $ v_k \Nto 0$, $ y_{k+1} \Nto y^\star$, and $ z_{k+1} \Nto z^\star$. Hence, Algorithm~\ref{bundlemethod} stops after finitely many iterations provided $\tol>0$.
Let $k_{\tol}$ be such a last iteration. By setting  $x_{\tol}=x_{k_\tol+1}$ and $\epsilon_\tol = \epsilon_{k_\tol +1}$ as in Proposition~\ref{prop:vk}, we conclude that $\dist (0, S(x_\tol,y^\star,z^\star))\leq \epsilon_\tol$ and $x_\tol$ is an near-stationary point to~\eqref{pbm}.
\mybox
\end{proof}

\subsection{Composite Proximal Bundle Method under Additional Structure}\label{sec:special}

Without further details about the actual problem~\eqref{pbm}, local models $\modelh_k$ for $h$ relying on cutting planes as in \eqref{cp} may be the more natural choice. This subsection assumes more structure about $h$, and then richer models become promising. We consider a structured subclass of \eqref{pbm} for which $h$ is itself a composition as detailed in the following refinement of Assumption~\ref{assumptions}.

\begin{assumption}{\rm (composite proximal bundle method; additional structure).}\label{assumptions-3layers}
Assumption~\ref{assumptions}(a,b,c,d) hold and is supplemented by the additional property: for convex nondecreasing $h_0\fc{\Re^d}{ \Re}$ and convex $h_i\fc{\Re^m}{ \Re}$, $i=1,\ldots,d$, the function $h$ can be written as the composition
\[
h = h_0 \circ H, ~~~\text{ where } H = (h_1, \dots, h_d).
\]
\end{assumption}

The new part of Assumption~\ref{assumptions-3layers} is motivated by situations where $h_0$ is a ``simple'' convex function but $h_1, \dots, h_d$ are more complex and only accessible via a first-order oracle; see \cite{Ferris_Huber_Royset_2024} for examples. Then, it could be advantageous to concentrate the local model building on the latter functions and leave $h_0$ intact. We now exploit this structure to provide a novel model $\modelh_k$, not of the cutting-plane kind in \eqref{cp}, that is economical in terms of computational memory usage and satisfies conditions~\eqref{model_req}.

At iteration $k$ of Algorithm~\ref{bundlemethod}, for $i=1,\ldots,d$,  let $B_i^{k} \subset \col{0,\ldots,k}$  be index sets and consider the cutting-plane models for $h_i$ (with notation in~\eqref{zjj}, $s_{ij} \in \partial h_i(z_j)$, and $\mathfrak{s}_{ik} \in \partial \check{h}^{k-1}_i(z_{k}) $):
\begin{subequations}\label{models}
\begin{equation}
   \check{h}_i^{k}(z)=
      \max \left\{
\begin{array}{cc}
     \check{h}_i^{k-1}(z_{k}) + \inner{\mathfrak{s}_{ik}}{z - z_{k} }  \\
   h_i(z_j)+\inner{s_{ij}}{z-z_j},\;  j \in B_i^{k}
\end{array}
   \right\}\leq h_i(z)\quad \forall\, z \in \Re^m. \label{models-h}
\end{equation}
(At $k=0$, the first linearization above is absent.)
 Accordingly, we define
\begin{equation}\label{models-H}
\check{H}_k(z)=(\check{h}_1^k(z),\ldots, \check{h}_d^k(z))\quad \mbox{and}\quad \modelh_k(z)=h_0\circ \check{H}_k(z)
\end{equation}
\end{subequations}
as models for the mapping $H$ and composition $h=h_0\circ H$, respectively.
As a result, the master program~\eqref{spbm} can be recast as the following strongly convex problem:
\begin{align}
\nnmin_{x\in X,r\in\reals^d} ~f_0(x)+ h_0(r) + \frac{1}{2t_k}\norm{x-\hat x_k}^2 & \label{aux3}\\
 \mbox{subject to } ~~\check{h}_i^k\big(F(\hat x_k)+\nabla F(\hat x_k)(x-\hat x_k)\big) & \leq r_i,\; i=1,\ldots,d.\nonumber
\end{align}
Since $h_0$ is ``simple,'' this master problem is presumably solvable efficiently. This is the case when $h_0$ is smooth or piecewise linear and expressed by a moderate number of explicitly known pieces or in the form of a real-valued Rockafellar-Wets monitoring function (see, e.g., \cite{Royset.25}) for which there is a convenient minimization formula that can be incorporated into \eqref{aux3}.

The proof of the following proposition can be found in the Appendix. 

\begin{proposition}{\rm (local model under additional structure).}\label{lem:doublecomposition}
Under Assumption~\ref{assumptions-3layers}, consider the local models in~\eqref{models} and suppose that, for every iteration $k$, the index sets $B_i^{k+1}$ are updated as follows: 
\begin{itemize}
\item after a serious step, choose $B_i^{k+1}$ arbitrarily such that  $\col{k+1} \in B_i^{k+1}  \subset \col{0,\ldots,k+1}$; 

\item after a null or backtracking step, choose $B_i^{k+1}$ such that $\{k+1,  \hat j\} \subset B_i^{k+1}$, where  $\hat j$ is the iteration index of the last serious iterate.
\end{itemize}
Then the model $\modelh_k$ in~\eqref{models} satisfies~\eqref{model_req} and~\eqref{uniform_bound}.
\end{proposition}

We remark that after a serious step,  the aggregate linearization $\check{h}_i^k(z_{k+1}) + \inner{\mathfrak{s}_{ik+1}}{z - z_{k+1} } $ need not be present in the definition of $\check h^{k+1}_i(z)$.

\begin{example}{\rm (buffered failure probability constraints (cont.)).}\label{ex:buffered2}
We recall that Example~\ref{ex:buffered} satisfies Assumption~\ref{assumptions} under minor additional conditions as discussed after that assumption. An extension to the setting of Assumption~\ref{assumptions-3layers} is meaningful when the number of scenarios $s$ is extreme large. The function $h^\nu$ in this example can be written as $h^\nu = h_0 \circ H$, where 
\[
h_0(\gamma) = \rho^\nu\max\{0,\gamma\} ~~\text{ and } ~~ H(u) = \max_{\pi \in P_\alpha}  \sum_{i = 1}^s \pi_i \max_{j=1, \dots, q} u_{ij}.
\]
Then, $h_0$ is a simple, convex, and nondecreasing function and $H$ is convex but more complex and practically available only via an oracle.\mybox
\end{example}

\section{Local Model and Implementation: Composite DC Programming}\label{sec:dc}

In addition to Algorithm~\ref{bundlemethod} from the previous section, we provide two more alternatives for implementing Line 3 of Algorithm~\ref{algo}. Again we omit the superscript $\nu$ and develop the inner algorithms for the actual problem \eqref{pbm} under additional assumptions. Still, a main use of the algorithms is on Line 3 of Algorithm~\ref{algo}: obtain near-stationary points of the approximating problems \eqref{pbm-a}, which can be constructed to satisfy the more restrictive assumptions. 

In contrast to Algorithm \ref{bundlemethod}, which is a bundle method, we now dispose of the smoothness assumption for the mapping $F$. This has the advantage that a lLc mapping in the actual problem \eqref{pbm} does not need to be replaced by a smooth global approximation in \eqref{pbm-a}. Instead we rely on dc structure. 
Consider now problem~\eqref{pbm} under the following assumption:
\begin{assumption}{\rm (composite dc programming).}\label{assumptions-dc}
\mbox{ }
\begin{enumerate}[label=(\alph*), start=1]
\item $X\subset \Re^n$ is a nonempty, closed, and convex.
\item $f_0\fc{\Re^n}{\Re}$ is smooth and  convex.
\item $F(x)=(f_1(x),\ldots,f_m(x))$, with each component function $f_i\fc{\Re^n}{\Re}$ being a dc function with given dc decomposition $f_i(x)=f_{i}^1(x)-f_{i}^2(x)$. With this notation, $f_{i}^1\fc{\Re^n}{\Re}$ and $f_{i}^2\fc{\Re^n}{\Re}$ are convex functions, and 
we take $D_i(x)=\partial f_i^1(x)-\partial_e f_i^2(x)$ in~\eqref{eqn:S}, with $e\geq0$ arbitrary but fixed.
\item $h\fc{\Re^m}{\Re}$ is  convex and nondecreasing. 
\end{enumerate}
\end{assumption}

Under this assumption, problem~\eqref{pbm} is a composite dc programming problem. We highlight that dc optimization has received significant attention is the last few decades due to its broad range of applications and favorable structure for numerical optimization. We refer the interested reader to the review paper \cite{LeThi_Tao_2018} and tutorial \cite{Oliveira_ABC_2020} on dc programming; see also the books \cite[\S\, 6.H, 6.I]{Royset_Wets_2021}, \cite[Chapter 7]{Cui_Pang_2022}, and \cite[\S\, 4.6]{WWBook}.

In general, the Clarke subdifferential calculus cannot be directly used to compute subgradients of dc functions, since $\bar \partial f_i(x) \subset \partial f_i^1(x)-\partial f_i^2(x)$, and this inclusion is strict in general. Equality holds when at least one of the component functions is smooth. In practice,  oracles for the dc components $f_i^1$ and $f_i^2$ are often readily available, whereas oracle for $f_i$ itself is not. For this reason, and in order to retain greater flexibility in the definition of oracles, Assumption~\ref{assumptions-dc} adopts
$D_i(x)=\partial f_i^1(x)-\partial_e f_i^2(x)$, a milder requirement. In our developments, $e\geq0$ is arbitrary but fixed. While we may set $e=0$ in the asymptotic analysis of Theorem~\ref{theo:dc} below, Theorem~\ref{theo:dc-epsilon} requires $e>0$ in order to analyze the approximate stationarity condition used in Algorithm~\ref{algo}.

In the univariate composition case (i.e., $m=1$), Assumption \ref{assumptions-dc}(d) can be relaxed to merely convexity. This follows from the fact that any univariate real-valued convex function can be written as the sum of two real-valued convex functions, one nondecreasing and the other nonincreasing; see \cite[Lemma 6.1.1]{Cui_Pang_2022}.
%

Certain compositions involving an outer concave function and an inner convex function remain compatible with our framework.
\begin{example}{ \rm (composition with an outer concave function).}
Let $\psi:\Re^n \to [0,\infty)$ be a convex function and $\varphi:[0,\infty)\to \Re$ a univariate, concave, and nondecreasing function such that $\varphi'_+(0)=\lim_{t \to 0+}(\varphi(t)-\varphi(0))/t < \infty$. The optimization problem $\min_{x \in X} f_0(x)+\varphi(\psi(x))$ 
 appears in applications where sparse
solutions are sought \cite{Oliveira_Tcheou_2019}. To reformulate this problem within our framework, consider an arbitrary constant $\tau \geq \varphi'_+(0)$ and the following dc decomposition $\varphi(\psi(x))= \tau \psi(x)-[\tau\psi(x)-\varphi(\psi(x))]=f^1(x)-f^2(x)$ (convexity of $f^2(x)=\tau\psi(x)-\varphi(\psi(x))$  follows from \cite[Proposition 4.11]{WWBook}.) We are thus in the setting of Assumption~\ref{assumptions-dc} with $h(z)=z$. \mybox
\end{example}
To tackle problem~\eqref{pbm} under Assumption~\ref{assumptions-dc} we may consider  Algorithm~\ref{algo-dc}.

\begin{algorithm}[htb]
\caption{Composite dc Algorithm}
\label{algo-dc}
\begin{algorithmic}[1]
{\footnotesize
\State Data: $x_0 \in X$, $t>0$, and $\tol\geq 0$
\For{$k= 0, 1, \dots$}
 \State Compute $(f_{i}^2(x_k), s_{ik}^2 \in \partial f_{i}^2(x_k))$, $i=1,\ldots,m$, and define 
 \[\bar F_k(x)=\begin{pmatrix}
    f_{1}^1(x)-[f_{1}^2(x_k)+\inner{s^2_{1k}}{x-x_k}]\\
    \vdots \\
    f_{m}^1(x)-[f_{m}^2(x_k)+\inner{s^2_{mk}}{x-x_k}]
\end{pmatrix}
\]
\State Compute $(x_{k+1}, y_{k+1}, z_{k+1})$ satisfying \label{line:xkk}
\begin{subequations}\label{eqn:MP-dc}
\begin{align}
z_{k+1}&=\bar F_k(x_{k+1}) \label{eqn:MP1-dc}\\
y_{k+1} &\in \partial h(z_{k+1})\label{eqn:MP2-dc}\\
-\nabla f_0(x_{k+1}) - \sum_{i=1}^m y_{ik+1} [\partial f_i^1(x_{k+1})-s^2_{ik}] -\frac{x_{k+1}- x_k}{t}  & \in N_X(x_{k+1})\label{eqn:MP3-dc} 
\end{align}
\end{subequations}
\State Set $v_k\gets f_0(x_k)+h(F(x_k)) -[f_0(x_{k+1})+h(z_{k+1}) ]$ 
\State Define $e_{k}\gets \displaystyle\max_{i=1,\ldots,m} \col{f_i^2(x_{k+1}) - [f_i^2(x_k)+\inner{s_{ik}}{x_{k+1}-x_k}]}$
\If{ $v_k\leq \tol$ and $e_k\leq \tol$ }
\State Stop and return $x_{k+1}$
\EndIf
\EndFor
}
  \end{algorithmic}
\end{algorithm}

The main calculations of the algorithm takes place on Line 4, where \eqref{eqn:MP-dc} gives an optimality condition to the convex master problem
\begin{equation}\label{MP-dc}
\nnmin_{x \in X} f_0(x)+h(\bar F_k(x)) +\frac{1}{2t}\norm{x- x_k}^2.
\end{equation}
This problem can be solved by standard nonlinear optimization solvers provided all $f_i^1$ and $h$ are smooth functions. If $h$ is nonsmooth but of the max-type (e.g. $h(z)=\rho\max\col{0,z_1,\ldots,z_m}$, for some $\rho>0$), then the same holds after an epigraphical reformulation. Otherwise, if $h$, or at least one $f_i^1$, is a general nonsmooth convex function, then the master program~\eqref{MP-dc} can be solved by (convex) bundle methods such as those described in \cite[Chapter~11]{WWBook}.

The convexity of $f_i^2$ guarantees that the linearization error $e_k$ in Algorithm~\ref{algo-dc} is nonnegative. This error is then used by the algorithm as an additional stopping criterion. As a consequence, if the algorithm stops at iteration $k$ with $\tol \geq 0$, then the inclusion $s_{ik}^2 \in \partial f_i^2(x_k)$ together with the inequality $e_k \le \tol$ implies that $s_{ik}^2 \in \partial_{\tol} f_i^2(x_{k+1})$ for all $i = 1,\ldots,m$. When plugged into~\eqref{eqn:MP3-dc}, this inclusion yields an approximate stationary condition as we see in the next subsection.
\subsection{Convergence analysis}
Again, let $N^a \subset \nats$ be the index set of iterations generated by Algorithm \ref{algo-dc}, which may be finite if the algorithm stops after finitely many steps of simply be $\nats$ otherwise.
We start the convergence analysis of Algorithm~\ref{algo-dc} by observing that the constructed sequence of function values is nonincreasing. Indeed, as $h$ is nondecreasing and $F(\cdot)\leq \bar F_k(\cdot)$ due to the dc structure, we have that
\begin{align}
 f_0(x_{k+1})+h(F(x_{k+1})) &\leq  f_0(x_{k+1})+h(F(x_{k+1})) + \frac{1}{2t}\norm{x_{k+1}- x_k}^2
 \nonumber \\
 &\leq  f_0(x_{k+1})+h(\bar F_k(x_{k+1})) +  \frac{1}{2t}\norm{x_{k+1}-x_k}^2 \label{eq:dc-aux1}\\
 &=  f_0(x_{k+1})+h(z_{k+1}) + \frac{1}{2t}\norm{x_{k+1}- x_k}^2\nonumber\\ 
&  =\min_{x \in X} f_0(x)+h(\bar F_k(x)) + \frac{1}{2t}\norm{x- x_k}^2 \nonumber\\
& \leq f_0(x_{k})+h(\bar F_k(x_{k}))= f_0(x_{k})+h(F(x_{k})),\nonumber
\end{align}
showing that $\col{f_0(x_{k})+h(F(x_{k})), k\in N^a}$ is a nonincreasing sequence and
\[
 \frac{1}{2t}\norm{x_{k+1}- x_k}^2\leq v_k \quad \forall\,k \in N^a.
\]
Observe that if $v_k =0$, then $x_k$ solves~\eqref{MP-dc}. Hence, when $v_k =0$, we have that $x_{k+1}=x_k$, $e_k=0$, and \eqref{eqn:MP-dc} gives $z_{k+1}=\bar F_k(x_{k+1})=F(x_k)$ and $y_{k+1} \in \partial h(F(x_k))$ satisfying the optimality condition $0 \in S(x_{k+1},y_{k+1},z_{k+1})$; recall that $D_i(x)=\partial f_i^1(x)-\partial_e f_i^2(x)$ and $e\geq 0$ by Assumption~\ref{assumptions-dc}(c).

\begin{theorem}\label{theo:dc}{\rm (convergence of Algorithm \ref{algo-dc}).}
    Consider Algorithm~\ref{algo-dc} with $\tol=0$ applied to problem~\eqref{pbm} under Assumption~\ref{assumptions-dc}. 
    Let $N^a\subset \nats$ denotes the index set of iterates generated by the algorithm and suppose that~\eqref{pbm} is level-bounded for $x_0$. If $N^a$ is infinite, then $\lim_{k \in N^a}v_k = 0$, $\nliminf_{k \in N^a}e_k= 0$, and any cluster point $x^\star$ of the sequence $\{x_k,k\in N^a\}$ is stationary, i.e., $0 \in S(x^\star,y^\star,z^\star)$ for some $y^\star$ and $z^\star=F(x^\star)$. If $N^a$ is finite, then the above holds with $x^\star$ being the last element of the collection $\col{x_k, \in N^a}$.
\end{theorem}
\begin{proof}
If the algorithm stops at iteration $k$ with $\tol=0$, then $x_{k+1}=x_k$, $v_k=e_k=0$, and $x_{k+1}$ is stationary as noted prior to the theorem. Let us now suppose that the algorithm loops forever, i.e., $N^a=\nats$.  As $\col{f_0(x_{k})+h(F(x_{k})), k\in \nats}$ is a nonincreasing sequence, we get that 
\[\col{x_k, k\in \nats} \subset \col{x \in X:\, f_0(x)+h(F(x))\leq f_0(x_0)+h(F(x_0))}
\]
is  bounded and thus, there exists an index set $\mathcal{K}$ such that $\lim_{k  \in \mathcal{K}} x_{k} =x^\star \in X$. Therefore, continuity of the involved functions and monotonicity of $\col{f_0(x_{k})+h(F(x_{k})), k\in \nats}$ give 
\begin{align*}
f_0 ( x^\star) + h(F(x^\star))&=\lim_{k \in \mathcal{K}}\col{f_0(x_{k})+h(F(x_{k}))}=\lim_{k \in \mathcal{K}}\col{f_0(x_{k+1})+h(F(x_{k+1}))}\\
&=\lim_{k \to \infty}\col{f_0(x_{k+1})+h(F(x_{k+1}))}.
\end{align*}
 This also shows, together with~\eqref{eq:dc-aux1}, that $v_k \to 0$ and $\lim_{k  \in \mathcal{K}} x_{k+1}=\lim_{k  \in \mathcal{K}} x_{k} = x^\star$.
 As the subdifferentials $\partial f^2_i(x_k)$ are outer semicontinous and locally bounded, we get that $\lim_{k \in \mathcal{K}}e_k =0$. Furthermore, there exists an index set $N\subset \mathcal{K}$ such that
 \[
 \lim_{k \in N} s^2_{ik} = \bar s^2_i \in \partial f_i^2(x^\star), \quad i=1,\ldots,m,
 \]
 and thus,
\[
\lim_{k \in N} \bar F_k(x) = \bar F(x):=\begin{pmatrix}
    f_{1}^1(x)-[f_{1}^2(\bar x)+\inner{\bar s^2_{1}}{x-x^\star}]\\
    \vdots \\
    f_{m}^1(x)-[f_{m}^2(\bar x)+\inner{\bar s^2_{m}}{x-x^\star}]
\end{pmatrix}
\quad \mbox{for all $x$.}
\] 
 Furthermore, regardless of $x\in X$,
\begin{align*}
f_0 (x^\star) + h(F(x^\star))  & = \lim_{k \in N }
\;f_0 (x_{k+1}) + h(F(x_{k+1}))+\frac{1}{2t}\norm{x_{k+1}- x_k}^2 \\
& 
     \leq \limsup_{k \in N }
\;f_0 (x_{k+1}) + h(\bar F_k(x_{k+1}))+\frac{1}{2t}\norm{x_{k+1}- x_k}^2 \\
 & \leq \limsup_{k \in N }
\;f_0 (x) + h(\bar F_k(x))+\frac{1}{2t}\norm{x- x_k}^2  \\
&= f_0(x) +h(\bar F(x))+\frac{1}{2t}\norm{x-x^\star}^2,
\end{align*}
where the first inequality follows from monotonicity of $h$ and definition of $\bar F_k$, and the second is due to optimality of $x_{k+1}$ in the master program.
All in all, we conclude that
 \begin{align*}
 f_0 (x^\star) + h(F(x^\star ))   & \leq  \min_{x \in X} \; f_0(x) +h(\bar F(x))+\frac{1}{2t}\norm{x-x^\star}^2
  \leq f_0(x^\star) +h(\bar F(x^\star))
 = f_0 (x^\star) + h(F(x^\star)),
 \end{align*}
 i.e., $x^\star$ solves $\min_{x \in X}  f_0(x) +h(\bar F(x))+\frac{1}{2t}\norm{x-x^\star}^2$. 
 Its optimality conditions reads as
 \begin{align*}
z^\star =\bar F(x^\star), \;
 y^\star \in \partial h(z^\star), \;
0  \in \nabla f_0(x^\star) + \sum_{i=1}^m y^\star_{i} [\partial f_i^1(x^\star)-\bar s^2_{i}]  +N_X(x^\star). 
\end{align*}
 Recall that $\bar F(x^\star)= F(x^\star)$ and Assumption~\ref{assumptions-dc}(c) ensures that $\partial f_i^1(x^\star)-\bar s^2_{i} \subset \partial f_i^1(x^\star)-\partial f^2_i(x^\star) \subset  D_i(x^\star)$. Consequently, $0 \in S(x^\star,y^\star,z^\star)$. \mybox
 \end{proof}
We stress that $e \ge 0$ in Assumption~\ref{assumptions-dc}(c) can be taken as $e = 0$ in the  asymptotic analysis of Algorithm~\ref{algo-dc} given above. However, in order to analyze the practical case in which Algorithm~\ref{algo-dc} is called by Algorithm~\ref{algo} at the outer iteration $\nu \in \nats$ with $\tol = \tol^\nu > 0$, we need to consider $e > 0$ so that the approximate stationary condition~\eqref{eqn:tol} can be shown to be satisfied with a vanishing error $\epsilon^\nu$. Consequently, the requirement $e > 0$ becomes crucial in the subsequent analysis.
\begin{proposition}\label{prop:dc}{\rm (approximate stationary).}
Under the assumption of Theorem~\ref{theo:dc}, let $N\subset \nats$ such that
$
 x_{k} \Nto  x^\star $ and
\[
\epsilon_{k+1}^2 =  \norm{F(x_{k+1})-z_{k+1}}^2+\frac{1}{t^2}\norm{x_{k+1}-x_k}^2+\dist^2\Big(0,\sum_{i=1}^my_{ik+1}[s^2_{ik} -\partial_{e} f^2_i(x_{k+1})\Big).
\]
Then
$\epsilon_{k+1} \geq  \dist(0,S(x_{k+1},y_{k+1},z_{k+1}))$ for all $k$. If $e>0$, then  $\lim_{k \in N} \epsilon_{k+1} =0$.
\end{proposition}
\begin{proof}
The proof of Theorem~\ref{theo:dc} shows that there exists an index set $N$ such that
\begin{subequations}\label{eq:dc-aux2}
\begin{align}
\lim_{k\in N} x_{k+1} = \lim_{k\in N} x_{k} & = x^\star \\
\lim_{k\in N} z_{k+1} =  \lim_{k \in N} \bar F_k(x_{k+1}) & = F(x^\star)=z^\star\\
\lim_{k\in N} y_{k+1} &= y^\star \in \partial h(F(x^\star)). 
\end{align}
\end{subequations}
Observe that
\begin{align*}
\dist^2\Big(0,S(x_{k+1},y_{k+1},z_{k+1})\Big) \leq \norm{F(x_{k+1})-z_{k+1}}^2
+\dist^2\Big(y_{k+1},\partial h(z_{k+1})\Big)  \\
\quad +\dist^2\Big(0,\nabla f_0(x_{k+1})+\sum_{i=1}^my_{ik+1}[\partial f^1_i(x_{k+1})-\partial_{e} f^2_i(x_{k+1})]+N_X(x_{k+1})\Big).
\end{align*}
It follows from~\eqref{eqn:MP2-dc} that $\dist^2(y_{k+1},\partial h(z_{k+1})) =0$ by construction.  
 Let us now work out the last term in the above inequality. Recall that \eqref{eqn:MP3-dc} gives
$
v_{k+1}=-\nabla f_0(x_{k+1}) - \sum_{i=1}^m y_{ik+1} (s^1_{ik+1}-s^2_{ik})-\frac{1}{t}(x_{k+1}-x_k)  \in N_X(x_{k+1}).
$
Therefore,
  \begin{align*}
&\dist^2\Big(0,\nabla f_0(x_{k+1})+\sum_{i=1}^my_{ik+1}[\partial f^1_i(x_{k+1})-\partial_{e} f^2_i(x_{k+1})]+N_X(x_{k+1})\Big)\\
&\leq\dist^2\Big(0,\nabla f_0(x_{k+1})+\sum_{i=1}^my_{ik+1}[\partial f^1_i(x_{k+1})-\partial_{e} f^2_i(x_{k+1})]+v_{k+1}\Big)\\
&=\dist^2\Big(0,\sum_{i=1}^my_{ik+1}\big[\partial f^1_i(x_{k+1})-s^1_{ik+1}-[\partial_{e}f^2_i(x_{k+1})-s^2_{ik}]\big]-\frac{1}{t}(x_{k+1}-x_k)\Big)\\
&\leq \dist^2\Big(0,\sum_{i=1}^my_{ik+1}[s^2_{ik} -\partial_{e} f^2_i(x_{k+1})\Big)+\frac{1}{t^2}\norm{x_{k+1}-x_k}^2.
\end{align*} 
(The first equality above follows from definition of $v_{k+1}$, and the last inequality because $0 \in \partial f^1_i(x_{k+1})-s^1_{ik+1}$.)
We have thus shown that
$\epsilon_{k+1} \geq  \dist(0,S(x_{k+1},y_{k+1},z_{k+1}))$. 

Let us now consider the case in which $e>0$ to show that $\lim_{k \in N} \epsilon_{k+1} =0$.
To this end,  first note that $s^2_{ik} \in \partial f_i^2(x_k)$ implies that $s^2_{ik} \in \partial_{e_{ik+1}} f_i^2(x_{k+1})$, with  $e_{ik+1}=f_i^2(x_{k+1}) - [f_i^2(x_k)+\inner{s_{ik}}{x_{k+1}-x_k}]\geq 0$. Furthermore, it follows from~\eqref{eq:dc-aux2} that $\lim_{k \in N} e_{ik+1} =0$, for all $i=1,\ldots,m$. As a result, given $e>0$, there exists $\tilde k\in \nats$  such that $0\leq e_{ik+1} \leq e$ for all $k\geq \tilde k$. Hence, the inclusion $s^2_{ik} \in \partial_{e} f_i^2(x_{k+1})$ holds for all $k \geq \tilde k$ and we get  that 
$\lim_{k \in N} \epsilon^2_{k+1} =\lim_{k \in N} \norm{F(x_{k+1})-z_{k+1}}^2+\frac{1}{t^2}\norm{x_{k+1}-x_k}^2=0$ thanks to~\eqref{eq:dc-aux2}.
\mybox
\end{proof}
Since the error $\epsilon_{k+1}$ is guaranteed to vanish when $e > 0$, the next result shows that Algorithm~\ref{algo-dc} computes points that are near-stationary provided that $\tol > 0$ and $e \ge \tol$. This latter inequality ensures that once the linearization error $e_k$ computed by the algorithm becomes smaller than $\tol$ (which occurs after finitely many steps thanks to Theorem~\ref{theo:dc} and Proposition~\ref{prop:dc}), the subgradient $s^2_{ik} \in \partial f_i^2(x_k)$ also belongs to the approximate subdifferential of $f_i^2$ at $x_{k+1}$, that is,
$
s^2_{ik} \in \partial_\tol f_i^2(x_{k+1}).
$
As a consequence, the condition $e \ge \tol$ guarantees that the last term in the definition of $\epsilon_{k+1}^2$ becomes zero, thereby significantly simplifying the analysis.
 \begin{theorem}\label{theo:dc-epsilon}{\rm (convergence of Algorithm \ref{algo-dc}, positive tolerance).}
Consider Algorithm~\ref{algo-dc} with $\tol>0$ applied to problem~\eqref{pbm} under Assumption~\ref{assumptions-dc}.
 Suppose that~\eqref{pbm} is level-bounded for $x_0$ and $e\geq \tol$ in Assumption~\ref{assumptions-dc}(c).
 Then the algorithm stops after finitely many iterations, say $k_\tol$, and $ x_\tol=x_{k_\tol+1}$ satisfies
\[
\dist\big(0, S(x_{\tol},y_{\tol} ,z_{\tol})\big) \leq \epsilon_\tol,
\]
for some $y_\tol \in \Re^m$, $z_\tol \in \Re^n$, and $\epsilon_\tol=\epsilon_{k_\tol+1}\geq 0$ given in Proposition~\ref{prop:dc}. Furthermore, it follows that $\epsilon_\tol$ vanishes as $\tol \to 0$.
\end{theorem}
 \begin{proof}
It follows from Theorem~\ref{theo:dc}, the positivity of $\tol$, and Proposition~\ref{prop:dc} that the algorithm stops after finitely many steps and $\dist(0,S(x_{k+1},y_{k+1},z_{k+1}))\leq \epsilon_{k+1}$ for all $k$. 
Let $k_{\tol}$ be this last iteration $k$ and denote
$ x_\tol=x_{k_\tol+1}$, $ y_\tol=y_{k_\tol+1}$, $ z_\tol=z_{k_\tol+1}$,  and $\epsilon_{\tol}= \epsilon_{k_{\tol}+1} $.
Let also $e_{ik+1}= f_i^2(x_{\tol}) - [f_i^2(x_{k_\tol})+\inner{s_{ik_\tol}}{x_{\tol}-x_{k_\tol}]}$ and $e_{k_\tol}=\max_{i=1,\ldots,m} e_{ik+1}$.  Therefore, the inclusion $s_{ik_\tol}^2 \in \partial f_i^2(x_{k_\tol})$ yields $s_{ik_\tol}^2 \in  \partial_{e_{k_\tol}} f_i^2(x_{\tol})\leq \partial_{\tol} f_i^2(x_{\tol})\leq \partial_{e} f_i^2(x_{\tol})$ because the stopping test ensures  $e_{k_\tol}\leq \tol$ and $e\geq \tol>0$ by assumption.
Hence, 
\begin{align*}
\epsilon_{k+1}^2 & =  \norm{F(x_{k+1})-z_{k+1}}^2+\frac{1}{t^2}\norm{x_{k+1}-x_k}^2+\dist^2\Big(0,\sum_{i=1}^my_{ik+1}[s^2_{ik} -\partial_{e} f^2_i(x_{k+1})\Big)\\
&= \norm{F(x_{k+1})-z_{k+1}}^2+\frac{1}{t^2}\norm{x_{k+1}-x_k}^2 \quad \mbox{at $k=k_\tol$.}
\end{align*}
As the right-hand side above vanishes when $\tol\to 0$ (recall~\eqref{eq:dc-aux2}), we conclude that $\epsilon_\tol$ vanishes as $\tol \to 0$. 
   \mybox
\end{proof}

In particular, when  Algorithm~\ref{algo} calls Algorithm~\ref{algo-dc} (with $\tol^\nu>0$) at the outer-iteration $\nu$, the latter returns a point $x^\nu$ ($=x_{k_{\tol^\nu}+1})$ satisfying~\eqref{eqn:tol} with  $(\epsilon^\nu)^2={\norm{F(x^\nu)-z^\nu}^2+\frac{1}{t^2}\norm{x^\nu-x_{k_\tol^\nu}}^2}$ and some $y^\nu$, $z^\nu$ provided $D^\nu_i(x)=\partial f^2_i(x)-\partial_{\tol^\nu} f^2_i(x)$. Note that $\epsilon^\nu \to 0$ and $\nOutLim D^\nu_i(x^\nu) \subset D_i(x) = \partial f^2_i(x)-\partial f^2_i(x)=\con D_i(x)$ whenever $\tol^\nu \to 0$ and $x^\nu\to x$.

\subsection{The Setting with Distance-to-Set Penalties}
The application with distance-to-set penalties in Example~\ref{example:dist2set} satisfies Assumption~\ref{assumptions-dc}. For $q\leq  m$, let $f_i(x)=\dist^2(x,K_i)$ be the
 squared (Euclidean) distance to the closed  set $K_i\neq \emptyset$, $i=1,\ldots,q$. 
 Assume that 
 $h\fc{\Re^m}{\Re}$ is of the form  $h(z) = \sum_{i=1}^q\frac{\rho_i}{2} z_i +  \sum_{i=q+1}^m \rho_i\max\col{0,z_i} $, with  $\rho_i>0$, $i=1,\ldots,m$, given penalties.
Observe that $h$ is nondecreasing and
\begin{equation}\label{eq:dist2set-dc}
\dist^2(x,K_i)=\min_{u \in K_i} \norm{u-x}^2 = \norm{x}^2-\displaystyle\max_{u \in K_i} \col{2\inner{x}{u} - \norm{u}^2}=f_i^1(x)-f_i^2(x).
\end{equation}
Suppose that $q=m$. In this case, problem~\eqref{pbm} reads as
$
\nnmin_{x \in X} \; f_0 (x) + \sum_{i=1}^m \frac{\rho_i}{2} \dist^2(x,K_i),
$
a setting studied in \cite{vanAckooij_Oliveira_2025} under a more general function $f_0$. 
In this case, Algorithm~\ref{algo-dc} can be simplified, leading to the following approach for computing a stationary point of the problem~\eqref{pbm} under these additional assumptions.
\begin{algorithm}[htb]
\caption{Proximal Distance Algorithm - a variant of Algorithm~\ref{algo-dc}}
\label{algo-dist2set}
\begin{algorithmic}[1]
{\footnotesize
\State Data: $x_0 \in X$, $t>0$,  and $\tol\geq 0$. Set $\mu\gets 1/t+\sum_{i=1}^m\rho_i$
\For{$k= 0, 1, \dots$}
 \State Compute $ p^i \in \proj_{K_i}(x_k)$, $i=1,\ldots,m$, and set $\hat p_k\gets \frac{1}{\mu}\Big(\sum_{i=1}^m\rho_i p^i_{k}+x_k/t\Big)$ 
\State Let $x_{k+1}$ be a solution of the problem 
\[\nnmin_{x \in X} \,f_0(x) +\frac{\mu}{2} \norm{x-\hat p_k}^2\]

\State Set $v_k\gets f_0(x_k)+ \sum_{i=1}^m \frac{\rho_i}{2}\dist^2(x_k,K_i)  -[f_0(x_{k+1})+\sum_{i=1}^m \frac{\rho_i}{2} \norm{x_{k+1}-p^i_{k}}^2]$ and $e_k$ as in Algorithm~\ref{algo-dc}
\If{ $v_k\leq \tol$ and $e_k\leq \tol$ }
\State Stop and return $x_{k+1}$
\EndIf
\EndFor
}
  \end{algorithmic}
\end{algorithm}

The solution set of the master program on Line 4 in Algorithm~\ref{algo-dist2set} coincides with the solution set of $\nnmin_{x \in X} \,f_0(x) +\sum_{i=1}^m \frac{\rho_i}{2}\norm{x-p_k^i}^2+\frac{1}{2t}\norm{x-x_k}^2$,
which is nothing other than the master problem~\eqref{MP-dc} under the assumption that $f_i(x)=\dist^2(x,K_i)$. 


\begin{acknowledgements}
The second author acknowledges financial support from the National Science Foundation under CMMI-2432337 and the Office of Naval Research under N00014-24-1-2492.

\end{acknowledgements}


\bibliographystyle{apalike}
\bibliography{references.bib}

\appendix

\section{Additional Proofs}

\begin{proof} {\bf of Proposition \ref{prop:optimality}.} Let $l_0(x,y) = f_0(x) + \langle F(x), y\rangle$ and $l_F(x,y) = \langle F(x), y\rangle$. Guided by \cite[Exercise 10.52]{Rockafellar_Wets_1998}, we find that there exists $y^\star\in\reals^m$ such that 
\[
0 \in \partial_x l_0(x^\star,y^\star) + N_X(x^\star), ~~~~~~y^\star \in \partial h\big(F(x^\star)\big)
\]
provided that the following qualification holds: 
\begin{equation}\label{eqn:qualVaAn}
y\in N_{\dom h}\big(F(x^\star)\big) ~~ \mbox{ and }~ ~  0 \in \partial_x l_F(x^\star,y) + N_X(x^\star)~~~\Longrightarrow~~~ y=0.
\end{equation}
By \cite[Corollary 10.9, Theorem 9.61]{Rockafellar_Wets_1998}, 
\[
\partial_x l_F(x, y) \subset \sum_{i=1}^m \partial (y_i f_i)(x) \subset \sum_{i=1}^m \con \big\{\partial (y_i f_i)(x)\big\} = \sum_{i=1}^m y_i \con \partial f_i(x)
\]
and after also invoking \cite[Proposition 4.58(c)]{Royset_Wets_2021}, one has
\[
\partial_x l_0(x, y) \subset \nabla f_0(x) + \sum_{i=1}^m y_i \con \partial f_i(x).
\]
Thus, the basic qualification \eqref{qualification} holding at $x^\star$ implies that \eqref{eqn:qualVaAn} holds as well because $\partial f_i(x^\star) \subset D_i(x^\star)$ for $i = 1, \dots, m$. Similarly, we conclude that  
\[
0 \in \nabla f_0(x^\star) + \sum_{i=1}^m y_i^\star \con D_i(x^\star) + N_X(x^\star), ~~~~~~y^\star \in \partial h\big(F(x^\star)\big).
\]
With $z^\star = F(x^\star)$, it follows that $0 \in S(x^\star, y^\star, z^\star)$. 

The second conclusion about equivalence with $0\in \partial(\iota_X + f_0 + h \circ F)(x^\star)$ holds by the following argument. Let $\tilde \phi(x) = \iota_X(x) + h(F(x))$. For $\tilde h:\reals^n\times\reals^m\to \Reals$ and $\tilde F:\reals^n\to \reals^{n+m}$ given by
\[
\tilde h(z_0,z) = \ind_X(z_0) + h(z)~~~\mbox{ and } ~~~ \tilde F(x) = \big(x, F(x)\big), 
\]
we obtain that $\tilde\phi(x) = \tilde h(\tilde F(x))$. Clearly, $\tilde F$ is lLc and $\tilde h$ is lsc and proper. Let $l_{\tilde F}(x,(w,y)) = \langle \tilde F(x), (w,y)\rangle$.  Under a qualification, \cite[Theorem 6.23]{Royset_Wets_2021} confirms that
\[
\partial \tilde\phi(x^\star) = \bigcup_{(w,y)\in \partial \tilde h(\tilde F(x^\star))} \partial_x l_{\tilde F}\big(x^\star,(w,y)\big)
\]
as long as $\tilde h$ is epi-regular at $\tilde F(x^\star)$ and $l_{\tilde F}(\cdot\,,(w,y))$ is epi-regular at $x^\star$ for all $(w,y) \in \partial \tilde h(\tilde F(x^\star))$. A closer examination shows that the basic qualification holding at $x^\star$ ensures that the needed qualification of the theorem indeed holds. By \cite[Proposition 4.63]{Royset_Wets_2021}, we conclude that $\tilde h$ is epi-regular at $\tilde F(x^\star)$ because $X$ is Clarke regular at $x^\star$ and $\partial h(F(x^\star)) \neq \emptyset$. Moreover, $l_{\tilde F}(\cdot\,,(w,y))$ is epi-regular at $x^\star$ because $y_i f_i$ is assumed to be epi-regular at $x^\star$; see \cite[Example 4.70]{Royset_Wets_2021}. 

By \cite[Proposition 4.58(c)]{Royset_Wets_2021}, $\partial (\iota_X + f_0 + h \circ F)(x) = \partial \tilde\phi(x) + \nabla f_0(x)$. Thus, it only remains to untangle the expression for $\partial \tilde\phi(x^\star)$. We obtain that
\begin{align*}
0 \in \partial (\iota_X + f_0 + h \circ F)(x^\star) ~~ &\Longleftrightarrow ~-\nabla f_0(x^\star) \in \partial \tilde \phi(x^\star)\\
 & \Longleftrightarrow ~ \exists y \in \partial h\big(F(x^\star)\big), ~~ -\nabla f_0(x^\star) \in \sum_{i=1}^m  \partial (y_i f_i)(x^\star) + N_X(x^\star),
\end{align*}
where we again appeal to \cite[Example 4.70]{Royset_Wets_2021} to compute $\partial_x l_{\tilde F}(x^\star,(w,y))$. Since $y_if_i$ is epi-regular at $x^\star$, $\partial (y_if_i)(x^\star)$ is a convex set and we conclude that
\[
\partial (y_i f_i)(x^\star) = \con \partial (y_i f_i)(x^\star) = y_i \con \partial f_i(x^\star) = y_i \con D_i(x^\star),
\]
which implies the assertion. \qed
\end{proof}

Before proving Proposition~\ref{lem:doublecomposition} we give some necessary ingredients. 

Note that
Assumption~\ref{assumptions-3layers} assures the existence of Lagrange multipliers $\tilde \lambda_i\geq 0$ satisfying the following KKT system, with $(x_{k+1},\tilde r)$ solving~\eqref{aux3}:
\begin{equation}\label{KKT}
\left\{
\begin{array}{llll}
0 \in & \displaystyle \nabla f_0(x_{k+1})+ \frac{(x_{k+1}-\hat x_k)}{t_k} + \nabla F(\hat x_k)^\top[\tilde\lambda_1 \mathfrak{ s}_{1k+1}+\ldots+ \tilde\lambda_d\mathfrak{s}_{dk+1}]+ N_X(x_{k+1})\\
0 \in & \partial h_0(\tilde r)-\tilde \lambda\\
0= &F(\hat x_k)+\nabla F(\hat x_k)(x_{k+1}-\hat x_k)-z_{k+1}\\
0 \in &  \partial \check{h}_i^k(z_{k+1})-\mathfrak{s}_{ik+1} \\
0= &\tilde\lambda_i [\check{h}_i^k(z_{k+1})- \tilde r_i] ,\;  i=1,\ldots,d\\
0\geq &\check{h}_i^k(z_{k+1})- \tilde r_i, \; \tilde\lambda_i \geq 0,\;  i=1,\ldots,d.
\end{array}
\right.
\end{equation}
Now we specify how to compute the models' subgradients  $\mathfrak{s}_{ik+1} \in \partial \check{h}_i^k(z_{k+1})$ in~\eqref{KKT}. To this end, we open up the cutting-plane models $\check{h}^k_i$ to write the master problem~\eqref{aux3} as:
\[
\left\{
\begin{array}{llll}
 \displaystyle \min_{x,r} &  \displaystyle f_0(x)+ h_0(r) + \frac{1}{2t_k}\norm{x-\hat x_k}^2\\
 \mbox{s.t.} &  x \in X, \,r\in \Re^d\\
  & \check h_1^{k-1}(z_k)+\inner{\mathfrak{s}_{1k}}{F(\hat x_k)+\nabla F(\hat x_k)(x-\hat x_k)- z_k)}\leq r_1\\
 & h_1(z_j)+\inner{s_{1j}}{F(\hat x_k)+\nabla F(\hat x_k)(x-\hat x_k)-z_{j}}\leq r_1,\; \forall\,j \in B_1^k\\
 & \qquad\qquad\qquad\vdots\\
  & \check h_d^{k-1}(z_k)+\inner{\mathfrak{s}_{dk}}{F(\hat x_k)+\nabla F(\hat x_k)(x-\hat x_k)- z_k)}\leq r_d\\
  & h_d(z_j)+\inner{s_{dj}}{F(\hat x_k)+\nabla F(\hat x_k)(x-\hat x_k)-z_{j}}\leq r_d,\; \forall\,j \in B_d^k.
\end{array}
\right.
\]
Again, Assumption~\ref{assumptions-3layers} assures the existence of Lagrange multipliers $\mu_{ij}, \mu_i\geq 0$ satisfying its KKT system:
\[
\left\{
\begin{array}{llll}
0 \in & \displaystyle \nabla f_0(x_{k+1})+ \frac{(x_{k+1}-\hat x_k)}{t_k} + \nabla F(\hat x_k)^\top\Big[\sum_{i=1}^m \Big(\mu_i\mathfrak{s}_{ik}+\sum_{j \in B_i^k} \mu_{ij}s_{ij}\Big)\Big]+ N_X(x_{k+1})\\
0 \in & \partial h_0(\tilde r)-\begin{pmatrix}
\displaystyle  \mu_1+\sum_{j \in B_1^k} \mu_{1j}\\
\vdots\\
\displaystyle   \mu_d+ \sum_{j \in B_d^k} \mu_{dj}\\
\end{pmatrix}\\
& \mbox{plus feasibility and complementarity conditions.}
\end{array}
\right.
\]
This system boils down to~\eqref{KKT} by taking
\begin{align}\label{aux2}
\
\tilde \lambda_i = \mu_i + \sum_{j \in B_i^k} \mu_{ij},\; \mbox{ and }\; \mathfrak{ s}_{ik+1}= \left\{  \begin{array}{lll}
\frac{\mathfrak{s}_{ik}}{\tilde \lambda} +\sum_{j \in B_i^k} \frac{\mu_{ij}}{\tilde\lambda_i}s_{ij} & \mbox{ if $\tilde\lambda_i>0$}\\
s_i \in \partial \check{h}_i^k(z_{k+1}) \mbox{ (arbitrary)} & \mbox{ if $\tilde \lambda_i=0$}
\end{array}
\right. \quad i=1,\ldots,d.
\end{align}

\begin{proof} 
 {\bf of Proposition~\ref{lem:doublecomposition}.}
Since $h_0$ is nondecreasing, and $\check{H}_k\leq H$ thanks to~\eqref{models-h}, we get that 
\[
\modelh_k(z) = h_0\circ \check{H}_k(z)\leq h_0\circ H(z)=h(z)\quad \forall\, z \in \Re^d,
\]
i.e., $\modelh_k$ is a lower model. Furthermore, as  $B_i^{k+1}$ contains the index of the linearization at the last serious iterate,
 we have by definition of $\check{h}_i^k$ that
$h_i(\hat z_k)+\inner{\hat s_{ik}}{z-\hat z_k}\leq \check{h}^k_{i}(z)$, for all $k$ (in particular~\eqref{model=c} holds).
We now turn our attention to the conditions in \eqref{model_req}.
Let us first focus on the third condition.
Given  $\mathfrak{s}_{i{k+1}} \in \partial  \check{h}_i^k(z_{k+1})$ in~\eqref{aux2}, $i=1,\ldots,d$, let 
\[
\mathbb{H}_{-k}(z)=
\begin{pmatrix}
\check{h}_1^k(z_{k+1})+\inner{\mathfrak{s}_{1k+1}}{z-z_{k+1}}\\
\vdots\\
\check{h}_m^k(z_{k+1})+\inner{\mathfrak{s}_{dk+1}}{z-z_{k+1}}
\end{pmatrix}
\leq \check{H}_{k}(z)\leq H(z)\quad \forall\, k.
\]
Observe that
\[
\partial  [h_0\circ \mathbb{H}_{-k}](z)= \col{\sum_{i=1}^d\lambda_i\mathfrak{s}_{ik+1}\,:\; \lambda \in \partial h_0(\mathbb{H}_{-k}(z))}.
\]
As $\mathbb{H}_{-k}(z_{k+1})=\check{H}_k(z_{k+1})$,
by taking $\tilde \lambda $ as in~\eqref{aux2} we conclude that
\[
y_{k+1}=\sum_{i=1}^d\tilde\lambda_i\mathfrak{s}_{ik+1} \in 
\partial  [h_0\circ \mathbb{H}_{-k}](z_{k+1}).
\]
Hence, by~\eqref{models-h}, we have that for all $z\in \Re^n$,
\begin{align*}
\modelh_{k+1}(z)&= h_0\circ \check{H}_{k+1}(z) \geq  h_0\circ \mathbb{H}_{-k}(z) \\
&\geq h_0\circ \mathbb{H}_{-k}(z_{k+1})+\inner{y_{k+1}}{z-z_{k+1}}\\
&= \modelh_k(z_{k+1})+\inner{y_{k+1}}{z-z_{k+1}}.
\end{align*}
This shows that the updated model overestimates the last linearization in~\eqref{model_req}.
Observe that, since $B_i^{k+1}$ contains the index of the linearization at the last serious iterate, the first condition in~\eqref{model_req} can be shown to hold in a similar manner as demonstrated above. The same applies to the second condition, as ${k+1} \in B_i^{k+1}$ for all $k+1$.

Now, to show that the proposed model satisfies~\eqref{uniform_bound} we only need to rely on convexity and the chain rule for differentiation:
as $h_0$ is convex and $\check{H}_k$ is polyhedral we have that
\begin{equation}\label{chain-rule}
\partial \modelh_k(z) = \partial  [h_0\circ \check{H}_{k}](z)= \col{\sum_{i=1}^d \lambda_i\mathfrak{s}_i\,:\; \lambda \in \partial h_0(r),\, r=\check{H}_{k}(z),\, \mathfrak{s}_i \in \partial \check{h}_i^k(z)}.
\end{equation}
As $\check{H}_k(z)=(\check{h}_1^k(z),\ldots, \check{h}_d^k(z))$,   
and the subdifferentials of $h_i$, $i=0,1,\ldots,d$, are locally bounded, we conclude that $\partial \modelh_k(z)$ is locally bounded and thus~\eqref{uniform_bound} follows. 
\mybox
\end{proof}

\end{document}